\documentclass[11pt,reqno]{amsart}
\usepackage{amsmath,amssymb,amsthm}
\usepackage{changepage,graphicx,float,hyperref,csquotes,tikz-cd}
\usepackage{cleveref}
\usepackage[letterpaper]{geometry}
\numberwithin{equation}{section}
\usepackage{amsthm}
\makeatletter
\def\paragraph{\@startsection{paragraph}{4}%
  \z@\z@{-\fontdimen2\font}%
  {\normalfont\bfseries}}
\makeatother

\usepackage{asymptote}
\newcommand{\der}{\operatorname{Der}}
\newcommand{\qstirB}{\operatorname{Stir}^B_q}
\newcommand{\stB}{\operatorname{st}^B}
\newtheorem{theorem}{Theorem}[section]
\newtheorem{definition}{Definition}[section]
\newtheorem{conjecture}[theorem]{Conjecture}
\newtheorem{corollary}{Corollary}[theorem]
\newtheorem{lemma}[theorem]{Lemma}
\title{The superspace coinvariant ring of type B}
\author{Sutanay Bhattacharya}
\begin{document}
\begin{abstract}
    Given the rank $n$ superspace $\Omega_n$, the ring of polynomial-valued differential forms on $\mathbb C^n$, one can define an action of hyperoctahedral group $\mathfrak B_n$ on it. This leads to a superspace coinvariant ideal $SR_n^B$, defined as the quotient of $\Omega_n$ by two-sided ideal generated by all $\mathfrak B_n$ invariants with vanishing constant terms. We derive the Hilbert series of $SR^B_n$ conjectured by Sagan and Swanson, and prove an operator theorem that yields a concrete description of the superharmonic space $SH^B_n$ associated to $SR^B_n$ as conjectured by Swanson and Wallach. We also derive an explicit basis of $SR^B_n$ using the theory of hyperplane arrangements.
\end{abstract}
\maketitle

\section{Introduction}
Let $n$ be a fixed positive integer, and let $\mathbb{C}[\mathbf x_n]$ denote the ring of polynomials over $\mathbb C$ in $n$ variables: $\mathbb{C}[\mathbf x_n]=\mathbb C[x_1,\ldots, x_n]$. There is a natural action of the symmetric group $\mathfrak S_n$ on $\mathbb{C}[\mathbf x_n]$ given by $\sigma(x_i)=x_{\sigma(i)}$. The \textit{coinvariant ideal} $I_n$ is generated by all $\mathfrak S_n$-invariants with vanishing constant term. This, together with the \textit{coinvariant ring} obtained by forming the quotient ring $R_n=\mathbb{C}[\mathbf x_n]/I_n$ are both important objects in algebraic combinatorics. E. Artin (\cite{artin1944galois}) showed that the monomials $\{x_1^{a_1}\cdots x_n^{a_n}:a_i<i\}$ in $\mathbb{C}[\mathbf x_n]$ descend to a basis of $R_n$, and one may calculate the Hilbert series of $R_n$ to be $$\operatorname{Hilb}(R_n;q)=[n]!_q$$ using this explicit basis. Here $[n]!_q$ denotes the $q$-factorial given by $$[n]!_q=[n]_q[n-1]_1\cdots [1]_q.$$where $[\cdot]_q$ refers to the $q$-analog of integers defined by $$[k]_q=1+q+\cdots+q^{k-1}.$$

This construction may be generalized to all \textit{pseudo-reflection groups}. These are defined as finite groups of unitary matrices generated by non-identity linear transformations of finite order that fix a hyperplane pointwise. Shepherd and Todd (\cite{shephard1954finite}) classified all pseudo-reflection groups: these include the symmetric group $\mathfrak S_n$ as well as the group of signed permutations $\mathfrak B_n$ as two notable special cases. One may define a corresponding coinvariant ring for each of these groups (see \cite{swanson2021harmonic} for details): properties of these rings are well-studied (\cite{chevalley1955invariants},\cite{borel1953cohomologie}).

In this article, we consider a similar construction on the \textit{superspace of rank n}, denoted $\Omega_n$, which is obtained by adjoining $n$ anti-commuting variables $\theta_1,\ldots, \theta_n$ to $\mathbb{C}[\mathbf x_n]$ that commute with the $x_i$'s. This may be thought of as the ring of all polynomial-valued differential forms on $\mathbb C^n$, where $\theta_i$ plays the role of $\mathrm{d}x_1$. Given any vector space $V$ and a subgroup $G\subseteq GL(V)$, $G$ has a natural action on $\Omega(V)$, the space of differential forms on $V$. In our case, this gives rise to an action of any pseudo-reflection group on $\Omega_n$ (details may be found in \cite{swanson2021harmonic}). This leads one to define a family of \textit{superspace coinvariant rings} as quotients of $\Omega_n$. 

Quotients of $\Omega_n$ have been the subject of significant research in recent years (\cite{zabrocki2019module},\cite{rhoades2022set},\cite{rhoades2024tutte},\cite{angarone2025superspace}). Of particular interest are the superspace coinvariant rings of type A and B. The first type is related to the action of $\mathfrak S_n$ on $\Omega_n$ given by $$\sigma(x_i)=x_{\sigma(i)},\quad \sigma(\theta_i)=\theta_{\sigma(i)}.$$Given this, the $\mathfrak S_n$-invariants of $\Omega_n$ with vanishing constant term generate the superspace coinvariant ideal (of type A), denoted $SI_n$, and the quotient $SR_n=\Omega_n/SI_n$ is called the superspace coinvariant ring of type A.

The second type corresponds to the action of $\mathfrak B_n$, the group of signed permutations. The elements of $\mathfrak B_n$ can be realized as permutations $\pi$ of the set $\{-n,\ldots, -1,1,\ldots, n\}$ with the property $\pi(-i)=-\pi(i)$ for every $i\in \{1,2,\ldots, n\}$. Given such a permutation $\pi$, we define its action on $\Omega_n$ by setting $$\pi(x_i)=x_{\pi(i)},\quad \pi(\theta_i)=\theta_{\pi(i)}$$as expected, with the added caveat that $x_{-i}$ and $\theta_{-i}$ are understood to denote $-x_i$ and $-\theta_i$ respectively. One may again define the coinvariant ideal $SI^B_n$ and $SR^B_n$ for this group action in similar fashion as before: we will revisit details of this construction in Section \ref{background}.

Sagan and Swanson conjectured \cite{sagan2024q} explicit monomial bases for $SR_n$ and $SR^B_n$, strengthening previous conjectures on their bigreaded Hilbert series (\cite{zabrocki2019module},\cite{swanson2021harmonic}) in terms of the so-called $q$-Stirling number and the type $B$ $q$-Stirling numbers respectively. Due to the failure of usual Gr\"obner theoretic straightening arguments, proofs of these conjectures have presented difficulties. Rhoades and Wilson (\cite{rhoades2024hilbert}) provided the first proof for the conjectured Hilbert series of $SR_n$ using sophisticated algebraic arguments. The proof for the monomial basis was found by Angarone, Commins, Karn, Murai, and Rhoades later (\cite{angarone2025superspace}), and used rather unexpected ideas from the theory of hyperplane arrangement, uncovering connections with the so-called Solomon-Terao algebras in the process.

The type B analogue is as follows. Define $q$-analogues of integers in the usual way: $$[k]_q=1+q+\cdots+q^{k-1}=\frac{q^{k}-1}{q-1},$$and given that, define the $q$-analogue of the double factorial as follows:$$[n]!!_q=[n]_q[n-2]_q[n-4]_q\cdots.$$The final term in the above product is either $[1]_q$ or $[2]_q$ depending on the parity of $n$. In \cite{sagan2024q}, Sagan and Swanson define the $B$-analogue of the $q$-Stirling numbers, denoted $\qstirB(n,k)$ by the following recurrence\begin{equation}
\qstirB(n,k)=\qstirB(n-1,k-1)+[2k+1]_q\cdot\qstirB(n-1,k)\end{equation}with the initial conditions 
\begin{equation}
\qstirB(0,k)=\begin{cases}
    1&k=0\\
    0&k\ne 0
\end{cases}
\end{equation}. 
\begin{conjecture}[{\cite[Conj. 1.19]{swanson2021harmonic}}]\label{conj} The bigraded Hilbert series for $SR^B_n$ is given by \begin{equation}
    \operatorname{Hilb}(SR^B_n;q,z)=\sum_{k=0}^n [2k]!!_q\cdot \operatorname{Stir}^B_q(n,k)\cdot z^{n-k}
\end{equation} where $q$ and $z$ tracks bosonic and fermionic degree, respectively.
    
\end{conjecture}Here $\operatorname{Stir}^B_q(n,k)$ denotes the type B analog of the so-called $q$-Stirling numbers that we will define in a later section. 

In this paper, we prove Conjecture \ref{conj} by extending the techniques in Rhoades-Wilson (\cite{rhoades2024hilbert}). Our main tool is the transfer principle used in \cite{rhoades2024hilbert}. This proceeds by reducing the problem of finding a basis for $SR^B_n$ in $\Omega_n$ to the problem of finding bases for a family of quotient rings in $\mathbb{C}[\mathbf x_n]$ cut out of certain colon ideals. This allows us to simplify a supercommutative algebra problem into a collection of commutative algebra problems, whence we are free to use classical techniques from commutative algebra. 

With this goal in mind, we start by defining a sequence of polynomials $$(p_{J,1},\ldots, p_{J,n})$$ for each $J\subseteq [n]$, which will yield Gale leading terms of carefully chosen elements of $SI^B_n$ (Lemma \ref{leadingterm}). We use these to come up with a spanning set for $SR^B_n$ (Lemma \ref{triangularitylemma}), which yields an upper bound for $\operatorname{Hilb}(SR^B_n;q,z)$.

To prove this is also a lower bound, we construct a family of operators $\mathfrak D_J$ with certain triangularity properties with respect to the Gale order (Lemma \ref{GaletriangularityLemma}). Similar to \cite[Lem. 5.2]{rhoades2024hilbert}, these considerations lead to a general method for producing bases of $SR^B_n$ using bases of certain colon ideals. Finally, we connect the polynomials $p_{J,i}$ to these colon ideals to prove that this method produces the appropriate lower bound for $\operatorname{Hilb}(SR^B_n;q,z)$.

As a consequence of these arguments, we will prove the `type B operator theorem' conjectured in \cite{swanson2021harmonic}. This is motivated by Haiman's operator theorem from the so-called ``diagonal coinvariants" (conjectured in \cite{haiman1994conjectures}, proved in \cite{haiman2002vanishing}). The key idea is to use an inner product to define the orthogonal complement of the coinvariant ideal, known as the \textit{harmonics} or \textit{inverse systems}, which is isomorphic to the coinvariant ring. This has the following advantage: the space of harmonics consists of actual polynomials rather than cosets of polynomials, which makes them easier to tackle. Swanson and Wallach conjectured that a version of this operator theorem holds for all complex reflection groups of the form $G(m,1,n)$ (the group of $n\times n$ matrices with exactly one non-zero entry in each row or column which is a $m^{\text{th}}$ root of unity, see \cite{shephard1954finite} for details). Rhoades and Wilson proved the case $m=1$ in which case $G$ equals the symmetric group $\mathfrak S_n$. In this paper (Theorem \ref{operatortheorem}), we prove the type B analog of this result, which corresponds to the case $m=2$ when $G$ is the group $\mathfrak B_n$. 

In a later section, we demonstrate an explicit basis for $SR^B_n$ (Corollary \ref{hyperplanebasiscor}). While this is different from the conjectured monomial basis \cite{sagan2024q}, this has the property that each basis element is a product of factors of the form $x_i$, $x_i\pm x_j$ and $\theta_i$. This connects the colon ideals considered in the earlier parts to the Solomon-Terao algebras of certain hyperplane arrangements.

The remainder of this paper is organized as follows. In Section \ref{background}, we define the main objects of study and review some necessary background material. In Section \ref{upperboundsection} we prove that the coefficients of the desired Hilbert series $\operatorname{Hilb}(SR^B_n;q,z)$ cannot exceed those of the conjectured expression by demonstrating a specific spanning set for $SR^B_n$. In Section \ref{acharacterizationofcolonideals} we obtain a characterization for some relevant colon ideals necessary for constructing bases of $SR^B_n$. In Section \ref{operatortheoremlowerbound}, we prove the operator theorem for type B and establish the claimed Hilbert series. Section \ref{hyperplanebasissection} applies the theory of hyperplane arrangements to derive an explicit basis for $SR^B_n$. Finally, in Section \ref{conclusion} we outline some ideas for further work in this area and questions that remain unresolved.

\section{Background\label{background}}
\subsection{Superspace and \texorpdfstring{$\mathfrak B_n$}{Bn}-action}
As indicated in the introduction, we define the \textit{superspace ring of rank} $n$ as $$\Omega_n=\mathbb C[x_1,\ldots, x_n]\otimes \wedge\{\theta_1,\ldots, \theta_n\}.$$ The variables $x_1,\ldots, x_n$ will often be called \textit{bosonic}, and monomials comprised of these will be called called \textit{bosonic monomials}. Similarly, we call the variables $\theta_1,\ldots, \theta_n$ \textit{fermionic}, and monomials comprised of these are called \textit{fermionic monomials}. Due to their anticommutativity, non-zero ferminonic monomials can always be expressed in the form (up to sign) $$\theta_{j_1}\theta_{j_2}\cdots\theta_{j_k}$$ for $1\le j_1<j_2<\cdots<j_k\le n$. For a set $J=\{j_1<j_2<\cdots<j_k\}$, we will abbreviate the monomial $\theta_{j_1}\theta_{j_2}\cdots\theta_{j_k}$ as $\theta_J$ for convenience.

For certain triangularity arguments, it will be useful to have a partial order on the subsets $J\subseteq[n]$ (and therefore on the monomials $\theta_J$. We will use the \textit{Gale order}. This is defined as follows: if $\{a_1<\ldots<a_r\}$ and $\{b_1<\cdots<b_r\}$ are subsets of $[n]$ of equal size, we say $$\{a_1<\ldots<a_r\}\le_{\text{Gale}}\{b_1<\cdots<b_r\}$$ if we have the inequalities$$a_1\le b_1,\;\;a_2\le b_2,\;\;\ldots,\;\; a_r\le b_r.$$

Let us define a few operators on $\Omega_n$ that will be useful later. The partial derivative operators $\partial_i$ for $1\le i\le n$ act as usual on $\mathbb C[\mathbf x_n]$. This action can be extended to $\Omega_n$ by letting $\partial_i$ treat the fermionic variables as constant. 

An analogous operator $\partial^{\theta}_i$ can be defined that acts on the second factor in the tensor product $\mathbb C[x_1,\ldots, x_n]\otimes \wedge\{\theta_1,\ldots, \theta_n\}$ and treats the bosonic variables as constant. This acts on the fermionic monomials as 

$$\partial^\theta_i:\theta_{j_1}\theta_{j_2}\cdots\theta_{j_r}\mapsto\begin{cases}
    (-1)^{s-1}\theta_{j_1}\cdots\widehat{\theta_{j_s}}\cdots \theta_{j_r} &\text{if }j_s=i\\
    0&\text{otherwise.}
\end{cases}$$The Euler operator $d$ acts on $\Omega_n$ via $$d(f)=\sum_{i=1}^n \partial_i(f)\cdot \theta_i.$$ For any integer $j\ge 1$, one can define higher versions of this operator  $d_j:\Omega_n\to\Omega_n$ as $$d_j(f)=\sum_{i=1}^n\partial_i^j(f)\cdot\theta_i.$$Given any set $J=\{j_1<j_2<\cdots<j_k\}$ of positive integers, we define $$d_J:=d_{j_1}d_{j_2}\cdots d_{j_k}$$which will be a useful shorthand for products of $d_j$'s.

We recall that the group $\mathfrak B_n$ consists of permutations $\pi$ of the set $\{-n,\ldots, -1,1,\ldots, n\}$ that satisfy $\pi(-i)=-\pi(i)$ for all $i\in [n]$. This acts on $\Omega_n$ as follows: $$\pi\cdot x_i=x_{\pi(i)}\qquad \pi\cdot \theta_i=\theta_{\pi(i)}$$where we let $x_{-i}=-x_i$ and $\theta_{-i}=-\theta_i$ for $i\in [n]$. It is easy to see that $p_{2i}$ for $1\le i\le n$ is invariant under this action, where $p_k$ denotes the usual power sum symmetric polynomial: $$p_k=\sum_{i=1}^nx_i^k.$$In addition, the $\mathfrak B_n$-action commutes with the action of $d$, which implies for any $\mathfrak B_n$-invariant $f$, $df$ is a $\mathfrak B_n$-invariant as well.

Let $SI^B_n$ denote the ideal of $\Omega_n$ generated by $\mathfrak B_n$ invariants with vanishing constant term. The above discussion implies $SI^B_n$ contains the elements $\{p_2,\ldots, p_{2n},dp_2,\ldots, dp_{2n}\}$; it is known \cite{solomon1963invariants} that these in fact generate all of $SI^B_n$. This leads us to define our main object of study:

\begin{definition}
    The \textit{superspace coinvariant ring of type B} is the quotient ring $$SR^B_n=\Omega_n/SI^B_n.$$
\end{definition}

It will be useful to define (polynomial) \textit{coinvariant ideal of type B} as well: this is the ideal in $\mathbb C[\mathbf{x}_n]$ generated by all $\mathfrak B_n$-invariants with vanishing constant term, and will be denoted $I^B_n$. We note that $I^B_n$ is generated by $\{p_2\ldots, p_{2n}\}$ as a $\mathbb C[\mathbf x_n]$-ideal.
\subsection{\texorpdfstring{$B$}{B}-analogues of \texorpdfstring{$q$}{q}-Stirling numbers}
As noted in the introduction, we aim to prove that the Hilbert series of $SR^B_n$ is given by the expression in Conjecture \ref{conj}. Recall that $\qstirB(n,k)$ refers to the $B$-analogue of the $q$-Stirling numbers defined by the recurrence:
\begin{equation}
\qstirB(n,k)=\qstirB(n-1,k-1)+[2k+1]_q\cdot\qstirB(n-1,k)\end{equation}with the initial conditions 
\begin{equation}
\qstirB(0,k)=[k=0].
\end{equation}
Here $[\cdot ]$ is the Iverson bracket: it evaluates to $1$ if the statement inside is true and 0 otherwise.

We will now state a combinatorial identity involving these numbers that will be useful later. For any $J\subseteq[n]$, define the \textit{$J$-staircase of type $B$} to be the sequence $(\stB_1(J),\ldots, \stB_n(J))$ where $\stB_i(J)$ for $i\in[n]$ is defined by 
\begin{equation}
   \stB_1(J)=[1\not\in J]\label{eqstB1}
\end{equation}
and 
\begin{equation}
    \stB_i(J)=\stB_{i-1}(J)+[i\not\in J]+[i-1\not\in J]\label{eqstB2}
\end{equation} for $1<i\le n$. 

The process of computing this sequence can be described as follows: the first number in the sequence denotes whether $1$ is missing in $J$. Next, for each $i$, we count how many of the numbers $\{i-1,i\}$ is missing from $J$, and we increment the previous number in the sequence by exactly that amount to get the $i^{\text{th}}$ entry. As an example, let $n=6$ and $J=\{2,5,6\}$. Then the $J$-staircase of type B is $(1,2,3,5,6,6)$. There is an increase of $2$ going from the third entry to the fourth, because exactly $2$ of the numbers $\{3,4\}$ is missing from $J$.

Summing up the equations \ref{eqstB2} for $i=2,\ldots, k$ together with \ref{eqstB1} we get 
\begin{align}
    \stB_k(J)&=2([1\not\in J]+[2\not\in J]+\cdots+[k-1\not\in J])+[k\not\in J] \nonumber\\
    &=2(|\{1,2,\ldots, k-1\}\setminus J|)+[k\not\in J]\label{eqstB3}\\
    &=2(n-|J\cup\{k,\ldots, n\}|)+[k\not\in J].\label{eqstB4}
\end{align}
In particular \ref{eqstB3} shows that $\stB_k(J)$ depends only on $k$ and set $J$, not on the ambient set $[n]$.
\begin{lemma}\label{combidentity}
For any positive integer $n$ and any $k\in\{0,\ldots, n\}$, the following identity holds:
    $$\sum_{\substack{J\subseteq [n]\\|J|=n-k}}\left(\prod_{j=1}^n[\stB_i(J)+1]_q\right)=[2k]!!_q\cdot \operatorname{Stir}^B_q(n,k).$$
\end{lemma}
This identity follows from the arguments in \cite{sagan2024q}; we provide a more direct proof below.
\begin{proof}
    One may use the definition of $\qstirB(n,k)$ to deduce that the right--hand side satisfies the recurrence $$[2k]!!_q\qstirB(n,k)=[2k]_q\cdot[2k-2]!!_q\qstirB(n-1,k-1)+[2k+1]_q\cdot[2k]!!_q\qstirB(n-1,k)$$for $n>1$, with the initial conditions $$[2k]!!_q\qstirB(1,k)=\begin{cases}
        [1]_q&k=0\\
        [2]_q&k=1\\
        0&\text{otherwise}.
    \end{cases}$$

    It suffices to show that the left-hand side also satisfies the same recurrence with the same initial conditions. Let us denote the left-hand side as $P(n,k)$.

Consider the case $n=1$. If $k=0$, $P(1,k)$ equals $$[\stB_1(\{1\})+1]_q=[1]_q.$$If $k=1$, we have $$P(1,0)=[\stB_1(\emptyset)+1]_q=[2]_q.$$Finally, if $k\not\in\{0,1\}$, $P(1,k)$ is the empty sum and hence zero. These agree with the above initial conditions for $[2k]!!_q\qstirB(n,k)$.

For $n>1$, not that the sum in $P(n,k)$ runs over all $J\subseteq[n]$ with $|J|=n-k$, which can be divided into two classes:
\begin{itemize}
    \item $n\not\in J$, in which case $J=J'$ for some $j'\subseteq[n-1]$ with $|J'|=(n-1)-(k-1)$. Per equation \ref{eqstB3}, $\stB_k(J)=\stB_k(J')$ for all $k\in[n-1]$, so the corresponding summand for $J$ equals the same for $J'$ times $[\stB_n(J)+1]_q=[2k]_q,$ contributing $[2k]_q\cdot [2k-2]!!_q P(n-1,k-1)$ to the total.
    \item $n\in J$, in which case $J=J'\cup\{n\}$ for any $J'\subseteq[n-1]$ with $J'=(n-1)-k.$ By the same logic as above, corresponding product equals the same for $J'$ times $[\stB_n(J)+1]_q=[2k+1]_q,$ contributing $[2k+1]_q\cdot [2k]!!_q P(n-1,k)$ to the total.
\end{itemize}
Thus we have $$P(n,k)=[2k]_q\cdot P(n-1,k-1)+[2k+1]_q\cdot P(n-1,k),$$which agrees with the recurrence for $[2k]!!_q\qstirB(n,k).$
\end{proof}


\subsection{Harmonics and superharmonics}\label{harmonicsection}

One of the most important tools in investigating coinvariant rings is the theory of harmonics, also known as Macaulay inverse systems. In this section, we briefly describe the main ideas behind this tool.

Given any $f=f(x_1,\ldots, x_n),g\in\mathbb C[\mathbf x_n]$, we can define $$f\odot g=f(\partial_1,\ldots, \partial_n)(g).$$This is well defined since the operators $\partial_i$ satisfy the defining relations that $x_1,\ldots, x_n$ satisfy: $\partial_i\partial_j=\partial_j\partial_i$. Using this, one can check that following defines an inner product: $$\langle -,-\rangle:\mathbb C[\mathbf x_n]\times\mathbb C[\mathbf x_n]\to\mathbb C\qquad \langle f,g\rangle=\text{constant term of }f\odot \overline{g},$$where $\overline{g}$ denotes the complex conjugate of $g$.

Given any homogeneous ideal $I$ in $\mathbb C[\mathbf x_n]$, this inner product enables us to form the orthogonal complement $I^\perp$, called the \textit{harmonic space} of $I$. We have the equality $$I^\perp=\{g\in\mathbb C[\mathbf x_n]:f\odot g=0\text{ for all }f\in I\}.$$ Further, we have the decomposition $\mathbb C[\mathbf x_n]=I\oplus I^\perp$ which implies the isomorphism $\mathbb C[\mathbf x_n]/I\cong I^\perp$ of graded $\mathbb C$-vector spaces .

For the specific case when $I$ is $I^B_n=(p_{2},\ldots, p_{2n})$, the ideal generated by $\mathfrak B_n$-invariants with vanishing constant term, $I^\perp$ has a simple description. Let $\delta^B_n$ denote the Vandermondian \cite{swanson2021harmonic} defined as $$\delta^B_n=\prod_{i=1}^nx_n\prod_{1\le i<j\le n}(x_i^2-x_j^2).$$Then we have the equality \cite{steinberg1964differential}$$(I^B_n)^\perp=\mathbb C[\mathbf x_n]\odot\delta^B_n.$$

We will denote $(I^B_n)^\perp$, the harmonic space attached to $I_n^B$, as $H^B_n$: this is isomorphic to the coinvariant ring $\mathbb C[\mathbf x_n]/I^B_n$ as graded $\mathfrak B_n$-modules. The following equivalence holds for any $f\in\mathbb C[\mathbf x_n]:$ $$f\in I_n\iff f\odot\delta^B_n=0.$$

A very similar set of objects can be defined for the superspace $\Omega_n$ as well. We begin by observing that the operators $\partial_i$ and $\partial^\theta_i$ defined above satisfy the defining relation of $\Omega_n$: $$\partial_i\partial_j=\partial_j\partial_i,\quad \partial^\theta_i\partial^\theta_j=-\partial^\theta_j\partial^\theta_i,\quad\partial_i\partial^\theta_j=\partial^\theta_j\partial_i.$$

This means any superspace element $f$ can be associated to an well-defined operator comprised of these by replacing $x_i\mapsto \partial_i$ and $\theta_i\mapsto\partial^\theta_i$. We call this new operator $\partial f$. Using this, we can define an action of $\Omega_n$ on itself: for $f,g\in\Omega_n$, we can define $$f\odot g=(\partial f)(g).$$

One can use this to define an inner product on $\Omega_n$ as before: we simply need to extend the conjugation operation to the fermionic variables. This can be done by reversing the monomials: $\overline{\theta_{j_1}\cdots\theta_{j_r}}=\theta_{j_r}\cdots\theta_{j_1}$. It can be checked that the pairing $$\langle -,-\rangle:\Omega_n\times\Omega_n\to\mathbb C\qquad \langle f,g\rangle=\text{constant term of }f\odot\overline{g}$$again defines an inner product. We can again define the orthogonal complement $$SH^B_n=(SI^B_n)^\perp=\{g\in\Omega_n:f\odot g=0\text{ for all }f\in I^B_n\}.$$This is again isomorphic to $SR^B_n=\Omega_n/SI^B_n$ as graded $\mathfrak B_n$-modules.

A related object is $SH'^B_n$, which can be defined as $$SH'^B_n=\mathbb C[\mathbf x_n]\odot\text{Span}_{\mathbb C}\{d_{2i_1-1}d_{2i_2-1}\cdots d_{2i_k-1}\delta^B_n:1\le i_1<i_2<\cdots<i_k\le n\}.$$In other words, $SH'^B_n$ is the smallest linear subspace of $\Omega_n$ that contains the Vandermondian $\delta^B_n$, is closed under the actions of $\partial_1,\ldots, \partial_n$ as well as the action of $d_{2i-1}$ for $1\le i\le n$. In \cite{swanson2021harmonic} Swanson and Wallach show that $SH'^B_n\subseteq SH^B_n$. We will show in Theorem \ref{operatortheorem} that in fact equality holds: $SH'^B_n=SH^B_n$, providing an explicit characterization of the superharmonic space.
\subsection{Hyperplane arrangements}
The theory of hyperplane arrangements will be crucial for us to derive an explicit basis for $SR^B_n$. Let us briefly review the key definitions and results we will need; some standard references are \cite{stanley2007introduction} and \cite{orlik2013arrangements}.

Given a vector space $V$, A \textit{hyperplane} in $V$ is a linear subspace of $V$ of codimension $1$. Any such hyperplane can be expressed as the zero set of a non-zero linear form $\alpha$: we will call this hyperplane $H_\alpha$. An \textit{hyperplane arrangement} is a finite set of hyperplanes $\mathcal A$.

Any hyperplane arrangement has an associated \textit{derivation module}. A \textit{derivation} of $\mathbb{C}[\mathbf x_n]$ is a $\mathbb C$-linear map $\delta:\mathbb{C}[\mathbf x_n]\to \mathbb{C}[\mathbf x_n]$ satisfying the Leibniz rule: $\delta(fg)=f\cdot\delta(g)+g\cdot\delta(f)$ for all $f,g\in \mathbb{C}[\mathbf x_n]$. 

The set $\der(\mathbb{C}[\mathbf x_n])$ of all derivations has a natural $\mathbb{C}[\mathbf x_n]$-module structure, and is generated by the partial derivative operators $\partial_1,\partial_2,\ldots,\partial_n$ as a free $\mathbb{C}[\mathbf x_n]$-module:$$\der(\mathbb{C}[\mathbf x_n])=\bigoplus_{1\le i\le n}\mathbb{C}[\mathbf x_n]\cdot\partial_i.$$These can be thought of as polynomial-valued vector fields on $\mathbb C^n$. A derivation is homogeneous of degree $d$ if it is of the form $\theta=f_1\partial_1+f_2\partial_2+\cdots+f_n\partial_n$, and each $f_i$ is a homogeneous polynomial in $\mathbb{C}[\mathbf x_n]$ with degree $d$.

Given a hyperplane arrangement $\mathcal A$ in $\mathbb C^n$, the set of derivations that represent vector fields parallel to all the hyperplanes in $\mathcal A$ form the derivation module $\der(\mathcal A)$. More concretely, $$\der(\mathcal A)=\{\delta\in\der(\mathbb{C}[\mathbf x_n]):\alpha\mid\delta(\alpha)\text{ for all }H_\alpha\in \mathcal A\}.$$It follows from this definition that for any $\delta\in \der(\mathcal A)$ and any $f\in\mathbb C[\mathbf x_n]$, we have $f\cdot\delta\in \mathbb C[\mathbf x_n]$; in other words, $\der(\mathcal A)$ is a $\mathbb C[\mathbf x_n]$-submodule of $\der(\mathbb C[\mathbf x_n])$.

The product of all the linear forms that define the hyperplanes in $\mathcal A$ is called the \textit{defining polynomial} of $\mathcal A$, denoted $Q(\mathcal A)$: $$Q(\mathcal A)=\prod_{H_\alpha\in \mathcal A}\alpha.$$The above condition can also be rephrased as $Q(\mathcal A)\mid \theta(Q(\mathcal A))$.

An arrangement $\mathcal A$ is called \textit{free} if $\der(\mathcal A)$ is a free $\mathbb{C}[\mathbf x_n]$-module. If this is the case, one can always find a set of homogeneous generators $\{\rho_1,\ldots, \rho_n\}$ that generate $\der(\mathcal A)$ as an $\mathbb{C}[\mathbf x_n]$-module. The degrees of the homogeneous generators are uniquely determined by the arrangement (up to ordering), and the multiset of these degrees is called the \textit{exponents} of $\mathcal A$ (denoted $\exp(\mathcal A)$).

Given a set of derivations of $\mathcal A$, the following criterion is often useful to prove $\mathcal A$ is free:

\begin{theorem}[{\cite[Theorem 4.23]{orlik2013arrangements}}]\label{saito}
    Suppose $\mathcal A$ is an arrangement in $\mathbb C^n$ and $\rho_1,\ldots,\rho_n$ are homogeneous elements of $\der(\mathbb{C}[\mathbf x_n])$ satisfying the following properties:
    \begin{enumerate}
        \item[(1)]  $\rho_1,\ldots,\rho_n\in\der(\mathcal A)$;
        \item[(2)] $\deg \rho_1+\cdots+\deg\rho_n=\left|\mathcal A\right|$;
        \item[(3)] $\rho_1,\ldots, \rho_n$ are linearly independent over $\mathbb{C}[\mathbf x_n]$.
    \end{enumerate}
    Then $\{\rho_1,\ldots, \rho_n\}$ is a free $\mathbb{C}[\mathbf x_n]$-basis of $\der(\mathcal A)$. In particular, $\exp(\mathcal A)=(\deg\rho_1,\ldots, \deg \rho_n)$
\end{theorem}

Given any hyperplane arrangement $\mathcal A$ and any hyperplane $H\in \mathcal A$, we can construct two new hyperplane arrangements. The first construction is the \textit{deletion} $\mathcal A\setminus H$: it is obtained by removing the hyperplane $H$ for $\mathcal A$:
$\mathcal A\setminus H=\{H'\in \mathcal A:H'\neq H\}.$ The second construction is the \textit{restriction} $\mathcal A^H$, obtained by taking intersections of the other hyperplanes of $\mathcal A$ with $H$. This yields a hyperplane arrangement in $H$ as follows:
$$\mathcal A^H=\{H'\cap H: H'\in\mathcal A, H'\ne H\}.$$

These constructions will be important to prove certain properties of hyperplanes inductively. In fact, the following result is often useful to establish freeness in inductively defined families of hyperplanes:

\begin{theorem}[Addition-Deletion Theorem, {\cite[Theorem 4.51]{orlik2013arrangements}}]\label{additiondeletion} 
Let $\mathcal A$ be a hyperplane arrangement in $\mathbb C^n$, and let $H\in \mathcal A$. Then any two of the following implies the third:
\begin{enumerate}
    \item[(1)] $\mathcal A$ is free with exponents $(e_1,\ldots, e_n)$.
    \item[(2)] $\mathcal A\setminus H$ is free with exponents $(e_1,\ldots, e_{n-1},e_n-1)$.
    \item[(3)] $\mathcal A^H$ is free with exponents $(e_1,\ldots, e_{n-1})$.
\end{enumerate}
\end{theorem}
\subsection{Solomon-Terao Algebras}
Given polynomials $p_1,\ldots, p_n\in \mathbb{C}[\mathbf x_n]$, one can define an $\mathbb{C}[\mathbf x_n]$-module map $\mathfrak c:\der(\mathbb{C}[\mathbf x_n])\to \mathbb{C}[\mathbf x_n]$ by setting $\partial_n\mapsto p_n$ and extending by $\mathbb{C}[\mathbf x_n]$-linearity. If the $p_i$'s are homogeneous of the same degree, we say that $\mathfrak c$ is homogeneous as well. For an arrangement $\mathcal A$ in $\mathbb C^n$, $\der(\mathcal A)$ is an $\mathbb{C}[\mathbf x_n]$-submodule of $\der(\mathbb{C}[\mathbf x_n])$, and so the image $\mathfrak c(\der(\mathcal A))$ is an ideal in $\mathbb{C}[\mathbf x_n]$. We can therefore form the quotient $\mathbb{C}[\mathbf x_n]/\mathfrak c(\der(\mathcal A))$: this is called the \textit{Solomon-Terao algebra} associated with $\mathcal A$ and $\mathfrak c$, and will be denoted $\mathcal{ST}(\mathcal A,\mathfrak c)$. The ideal $\mathfrak \mathfrak c(\der(\mathcal A))$ is called the \textit{Solomon-Terao ideal}, and will be denoted $\mathfrak c_\mathcal{A}$.

These algebras are well-behaved if we restrict our attention to free arrangements. In fact, the following result gives us a useful way to describe the Solomon-Terao ideal for a free subarrangement of a free arrangement. Let us recall that given an ideal $I\subseteq \mathbb{C}[\mathbf x_n]$ and $f\in \mathbb{C}[\mathbf x_n]$, the \textit{colon ideal} $I:f$ is defined as $$I:f=\{g\in \mathbb{C}[\mathbf x_n]:fg\in I\}.$$

\begin{theorem}[{\cite[Lem. 4.4]{angarone2025superspace}}]\label{arrcolonideal}
    Suppose $\mathcal A$ is a free arrangement in $\mathbb C^n$, and $\mathcal B\subseteq \mathcal A$ is a free subarrangement. Let $\mathfrak c:\der(\mathbb{C}[\mathbf x_n])\to \mathbb{C}[\mathbf x_n]$ be a homogeneous map so that $\mathcal{ST}(\mathcal A,\mathfrak c)$ is Artinian, and also that the polynomial $Q(\mathcal A)/Q(\mathcal B)$ does not belong to $\mathfrak c_\mathcal{A}$. Then $$\mathfrak c_\mathcal{B}=\mathfrak c_{\mathcal A}:\left(Q(\mathcal A)/Q(\mathcal B)\right).$$
\end{theorem}

We will exclusively use the special case where $\mathfrak c$ is the homogeneous map $\mathfrak a$ given by $\partial_i\mapsto x_i$. This has additional nice properties in the context of free arrangements: for example, it is known that for a free arrangement $\mathcal A$, $\mathfrak a_\mathcal{A}$ is a complete intersection and $\mathcal{ST}(\mathcal A,\mathfrak a)$ a finitely generated Poincar\'e duality algebra (see \cite{abe2020hessenberg}). In particular, we will never have to verify that $\mathcal{ST}(\mathcal A,\mathfrak a)$ is Artinian while applying the above theorem.

\section{Upper Bound\label{upperboundsection}}
The first part of our argument involves coming up with a spanning set of $SR^B_n$ in order to bound $\operatorname{Hilb}(SR^B_n;q,z)$ from above. For $S=\{i_1,\ldots, i_k\}\subseteq [n]$, define $h^2_r(S)$ via $$h_r^2(S):=h_r(x_{i_1}^2,\ldots, x_{i_k}^2)$$ where $h_r(t_1,\ldots, t_k)$ denotes the homogeneous symmetric polynomial of degree $r$ in the variables $t_1,\ldots, t_k$ obtained by summing all monomials of degree $r$ in $t_1,\ldots, t_k$. In other words, $h^2(S)$ is obtained from the usual homogeneous symmetric polynomials in $x_{I_1},\ldots, x_{i_k}$ by replacing each variable with its square. 

Given $J\subseteq [n]$, we define $q_{J,i}$ as follows:
\begin{equation}\label{qdefn}
     q_{J,i}=\begin{cases}
    h^2_{r_i}(J\cup\{i,\ldots, n\})\cdot \theta_J & i\notin J\\
    dh^2_{r_i}(J\cup\{i,\ldots, n\})\cdot \theta_{J-i} & i\in J
\end{cases}
\end{equation}
where $r_i=n-\left| J\cup\{i,\ldots, n\}\right|+1$.
We will need to look at the projections of their Gale leading terms, which we define as follows:

\begin{equation}\label{pdefn}
    p_{J,i}=\begin{cases}
    h^2_{r_i}(J\cup\{i,\ldots, n\}) & i\notin J\\
    \partial_i h^2_{r_i}(J\cup\{i,\ldots, n\}) & i\in J
\end{cases}
\end{equation}
with $r_i$ as before. It is easy to use \eqref{eqstB4} to see that for any $J\subseteq[n]$ and any $i\in [n]$, 
\begin{equation}
    \deg p_{J,i}=\stB_i(J)+1.
\end{equation}

For the purposes of illustration, let us use the example $n=4$ and $J=\{2,4\}$. Then the $J$-staircase of type B is $(1,2,3,4)$, and we have 
\begin{align*}
   p_{J,1}&= h^2_1(\{1, 2, 3, 4\})=x_1^2+x_2^2+x_3^2+x_4^2,\\
p_{J,2}&=\partial_2 h^2_2(\{2, 3, 4\})=\partial_2(x_2^4+x_3^4+x_4^4+x_2^2x_3^2+x_3^2x_4^2+x_2^2x_4^2)=4x_2^3+2x_2x_3^2+2x_2x_4^2,\\
p_{J,3}&=h^2_2(\{2, 3, 4\})=x_2^4+x_3^4+x_4^4+x_2^2x_3^2+x_3^2x_4^2+x_2^2x_4^2,\\
p_{J,4}&=\partial_4h^2_3(\{2,4\})=\partial_4(x_2^6+x_4^6+x_2^4x_4^2+x_2^2x_4^4)=6x_4^5+2x_2^4x_4+4x_2^2x_4^3.
\end{align*}

As we will show soon, the $q_{J,i}$'s defined above are elements of $SI^B_n$ with some convenient properties. To see a concrete example, let us take $n=4$ and $J=\{2,4\}$ again, and compute $q_{J,2}:$
$$q_{J,2}=dh^2_2(\{2,3,4\})\cdot\theta_{\{4\}}=(4x_2^3+2x_2x_3^2+2x_2x_4^2)\theta_2\theta_4+(4x_3^3+2x_2^2x_3+2x_3x_4^2)\theta_3\theta_4.$$Note that $p_{J,2}$ appears in this expression as the coefficient of $\theta_J$, and every other $\theta_K$ appearing in the expression corresponds to some $K$ bigger than $J$ in Gale order; in other words, $p_{J,i}\cdot\theta_J$ is the Gale-leading term of $q_{J,i}$. This is true in general, as seen in the following lemma:

\begin{lemma}\label{leadingterm}
For any $J\subseteq [n]$ and $i\in [n]$, $q_{J,i}$ belongs to $SI^B_n$. Furthermore, $q_{J,I}$ lies in $\bigoplus_{J\le_{\text{Gale}}K}\mathbb C[\mathbf x_n]\cdot \theta_K$, where the projection onto the $\mathbb C[\mathbf x_n]\cdot \theta_J$ component equals $\pm p_{J,i}\cdot \theta_J$.
\end{lemma}
\begin{proof}
It is well-known that for any $S\subseteq [n]$ with $r>n-|S|$, we have $h_r(S)\in I_n$ . Replacing the $x_i$'s with $x_i^2$'s, this implies $h_r(S)$ is an $\mathbb C[x_1^2,\ldots, x_n^2]$-linear combination of the generators of $I^B_n$, which implies $h^2_r(S)$ is in $I^B_n$. The first statement now follows from this observation and the fact that $SI^B_n$ is closed under the action of $d$. 

The second statement is clear when $i\not\in J$. When $i\in J$, we simply note that $$dh_{r_i}^2(J\cup\{i,\ldots, n\})\cdot\theta_{J-i}=\sum_{k\in J\cup \{i,\ldots, n\}}\partial_kh^2_{r_i}(J\cup\{i,\ldots, n\})\cdot\theta_k\cdot\theta_{J-i}.$$A monomial in the above sum vanishes when $k\in J-i$, and the remaining monomials all satisfy $J\le_{\text{Gale}}K$. The one corresponding to $\theta_J$ is exactly equal to $p_{J,i}\cdot\theta_i\cdot\theta_{J-i}$, proving our claim. 
\end{proof}
It so happens that the $p_{J,i}$'s defined above are well-behaved. Before we expand on this, let us recall that a sequence of polynomials $(f_1,\ldots, f_r)$ in $\mathbb C[\mathbf x_n]$ is called a regular sequence if for $i=1,\ldots, r$ the polynomial $f_i$ is a non-zero divisor in the quotient ring $\mathbb C[\mathbf x_n]/(f_1,\ldots, f_{i-1})$, and the quotient $\mathbb C[\mathbf x_n]/(f_1,\ldots, f_r)$ is non-zero. 

If $(f_1,\ldots, f_r)$ is a regular sequence of homogeneous polynomials (not necessarily of the same degree) in $\mathbb C[\mathbf x_n]$, then the quotient $\mathbb C[\mathbf x_n]/(f_1,\ldots, f_r)$ is a finite-dimensional graded vector space, and in addition we have 
\begin{equation}
    \operatorname{Hilb}\left(\mathbb C[\mathbf x_n]/(f_1,\ldots, f_r);q\right)=[\deg f_1]_q\cdots [\deg f_r]_q.\label{hilbregularsequences}
\end{equation}

The regularity of sequences of polynomials, under certain nice conditions, can be checked by analyzing the variety cut out by the ideal generated by them, as seen from the following well-known lemma.

\begin{lemma}
    Let $f_1,\ldots, f_n$ be homogeneous polynomials in $\mathbb C[\mathbf x_n]$ of positive degree. Then  $(f_1,\ldots, f_n)$ is regular if and only if $$\{\mathbf x\in\mathbb C^n:f_1(\mathbf x)=\cdots=f_r(\mathbf x)=0\}=\{\mathbf 0\}.$$
\end{lemma}

It will turn out that for each $J\subseteq [n]$, $(p_{J,1},\ldots, p_{J,n})$ is a regular sequence of 
homogeneous polynomials. We will use the above lemma to prove this; but before we do, we need an algebraic identity involving the expressions $h^2_r(S)$ that appear in the definition of $p_{J,i}$.
\begin{lemma}
     For any $S\subseteq[n]$ with $a\in S$ and $b\notin S $, and integer $r>1$, we have $$\partial_a h^2_r(S)=(x_a^2-x_b^2)\partial_a h^2_{r-1}(S\cup b)+2x_a h_{r-1}^2 (S\cup b).$$
\end{lemma}
\begin{proof}
    The right-hand side of the above equation equals $\partial_a\left((x_a^2-x_b^2)h_{r-1}^2(S\cup b)\right)$, so in order to show the difference of the two sides is $0$, it suffices to show $h^2_r(S)-(x_a^2-x_b^2)h_{r-1}^2(S\cup b)$ is independent of $x_a$. We will in fact show $$h_r(S)-(x_a-x_b)h_{r-1}(S\cup b)=h_r(S\cup b-a)$$whence the above claim will follow by replacing all $x_i$'s with $x_i^2$'s. This rearranges to $$h_r(S)+x_b h_{r-1}(S\cup b)=h_r(S\cup b-a)+x_ah_{r-1}(S\cup b)$$which is clear since both sides equal $h_r(S\cup b)$.
\end{proof}

To show that the variety cut out by $(p_{J,1},\ldots, p_{J,n})$ we will need to identify certain key elements in this ideal.
\begin{lemma}
     For $J\subseteq [n]$, define $\mathcal I_J=( p_{J,1},\ldots,p_{J,n})\subseteq \mathbb C[\mathbf x_n]$. Then we have $\partial_j h^2_{n-|J|+1}(J)\in \mathcal I_J$ for any $j\in J$.
\end{lemma}
\begin{proof}
    We use the notation $r_i$ as defined in $\ref{qdefn}$. Pick any $j\in J$; we have $p_{J,j}=\partial_j h^2_{r_j}(J\cup\{j,\ldots,n\})\in \mathcal I_J$. If $\{j,\ldots, n\}\subseteq J$, then we are already done. If not, we will induction to show that $\partial_j h^2_{r_k}(J\cup\{k,\ldots, n\})\in \mathcal I_J$ for every $j\le k\le n$. 

    The base case $k=j$ is what we already assumed; suppose the statement is true for a given $k<n$. If $k\in J$, then $J\cup \{k,\ldots, n\}=J\cup\{k+1,\ldots,n\}$ and $r_k=r_{k+1}$, so there is nothing to prove. Suppose instead that $k\notin J$, so that $r_{k+1}=r_k+1$. Since $\partial_j h^2_{r_k}(J\cup\{k,\ldots, n\})\in\mathcal I_J$, and $p_{J,k}=h^2_{r_k}(J\cup\{k,\ldots,n\})\in\mathcal I_J$, using Lemma 1 for $a=j$, $b=k$, $S=J\cup\{k+1,\ldots, n\}$ and $r=r_k+1$ implies $\partial_j h^2_{r_{k+1}}(J\cup\{k+1,\ldots, n\})\in\mathcal I_J$, finishing the induction step.

    In particular, we have $\partial_j h^2_{r_n}(J\cup n)\in \mathcal I_J$. If $n\in J$ then $r_n=n-|J|+1$ and $J\cup n=J$, so the lemma follows. If $n\notin J$, then $r_n=n-|J|$. Also, $p_{J,n}=h^2_{r_n}(J\cup n)\in \mathcal I_J$, so using Lemma 1 with $a=j$, $b=n$, $S=J$ and $r=r_n+1=n-|J|+1$, we conclude $\partial_j h^2_{n-|J|+1}(J)\in\mathcal I_J$ as before.
\end{proof}

Now we are ready to show the claimed regularity of the sequence $(p_{J,1},\ldots, p_{J,n})$.
\begin{lemma}\label{regularsequence}
    For any $J\subseteq[n]$, the sequence $(p_{J,1},\ldots, p_{J,n})$ is a regular sequence of polynomials.
\end{lemma}
\begin{proof}
    It is clear that this is a sequence of homogeoneous polynomials of positive degree, so it suffices to show that variety $V(p_{J,1},\ldots, p_{J,n})$ in $\mathbb C^n$ equals the singleton $\{\mathbf 0\}$.

    We begin by noting that according to Lemma 6.2 in \cite{swanson2021harmonic} with $m=2$ and $j=n-|J|+1$, the polynomials $\partial_j h^2_{n-|J|+1}(J)$ for $j\in J$ have no non-trivial common zero in $\mathbb C^J$. So if $(a_1,\ldots, a_n)$ is any point in the variety $V(p_{J,1},\cdots, p_{J,n})$, the previous lemma implies that $a_j=0$ for all $j\in 0$. Setting $x_j=0$ for $j\in J$, the polynomials $P_{J,i}$ for $i\not\in J$ yield a triangular set of positive-degree polynomials in the variables $\{x_i:i\not\in J\}$, and this implies the result.
\end{proof}

The polynomials $p_{J,1},\ldots, p_{J,n}$ quotient out a space from $\mathbb C[\mathbf x_n]$ whose basis will help us construct to a spanning set $SR^B_n$. The property of $p_{J,i}$ that makes this work is the fact that superspace elements of the form $f\cdot\theta_J$ (for $f\in\mathbb C[\mathbf x_n]$) can be expanded in terms of bases of the aforementioned quotient spaces in a triangular fashion. Because of the regularity proved in Lemma \ref{regularsequence} and equation \eqref{hilbregularsequences}, these quotients will turn out to have predictable Hilbert series. The following lemma makes these claims explicit.

\begin{lemma}\label{triangularitylemma}
    For each $J\subseteq [n]$, there exists a set of $\mathcal B_n(J)$ of homogeneous polynomials in $\mathbb C[\mathbf x_n]$ with the following properties:
    \begin{itemize}
        \item We have the following equality:
        $$\sum_{m\in \mathcal B_n(J)}q^{\deg m}=[\stB_1(J)+1]_q\cdots [\stB_n(J)+1]_q;$$
        \item Given any polynomial $f\in\mathbb C[\mathbf x_n]$, we have the following expansion:
    $$f\cdot \theta_J=\left(\sum_{m\in\mathcal B_{n}(J)}c_{f,m}\cdot m\cdot\theta_J\right)+g+\Sigma$$where $c_{f,m}\in\mathbb C$, $g\in SI^B_n$, and $\Sigma\in\bigoplus_{J<_{\text{Gale}}K}\mathbb C[\mathbf x_n]\cdot \theta_K.$
    \end{itemize} 
\end{lemma}
\begin{proof}
    Lemma \ref{regularsequence} and equation \eqref{hilbregularsequences} imply that the Hilbert series of the quotient ring $\mathbb C[\mathbf{x}_n]/(p_{J,1},\ldots,p_{J,n})$ is given by $[\deg p_{J,1}]_q\cdots [\deg p_{J,n}]_q$. Therefore, for any $J$, there exists a set $\mathcal B_n(J)$ of homogeneous polynomials in $\mathbb C[\mathbf{x}_n]$ that descends to a basis of $\mathbb C[\mathbf x_n]/(p_{J,1},\ldots, p_{J,n})$, so that their degree generating function is given by $$\sum_{m\in \mathcal B_n(J)}q^{\deg m}=[\deg p_{J,1}]_q\cdots [\deg p_{J,n}]_q=[\stB_1(J)+1]_q\cdots [\stB_n(J)+1]_q.$$
    
    For any $f\in\mathbb C[\mathbf x_n]$, we have $$f=\left(\sum_{m\in\mathcal B_{n}(J)}c_{f,m}\cdot m\right)+\sum_{1\le i\le n}a_ip_{J,i}$$for some constants $c_{f,m}$ and $a_i$. Multiplying by $\theta_J$ yields $$f\cdot\theta_J=\left(\sum_{m\in\mathcal B_{n}(J)}c_{f,m}\cdot m\cdot\theta_J\right)+\sum_{1\le i\le n}a_ip_{J,i}\cdot\theta_J.$$But we have seen before in Lemma \ref{leadingterm} that $p_{J,i}\cdot\theta_J$ is the Gale-leading term of $q_{J,i}$ (up to sign), which means the above can be rewritten as $$f\cdot\theta_J=\left(\sum_{m\in\mathcal B_{n}(J)}c_{f,m}\cdot m\cdot\theta_J\right)+\sum_{1\le i\le n}\pm a_iq_{J,i}+\Sigma$$where $\Sigma\in\bigoplus_{J<_{\text{Gale}}K}\mathbb C[\mathbf x_n]\cdot\theta_K$. Finally, $q_{J,i}\in SI^B_n$ from Lemma \ref{leadingterm}, so letting $g=\sum_{1\le i\le n}\pm a_iq_{J,i}$ as in the above sum, we obtain the desired expansion.
\end{proof}
We are now in a position to come up with a spanning set for $SR^B_n$ whose degree generating function will be easy to compute.
\begin{lemma}\label{hilbertspanning}
    The set $$\mathcal B_n=\bigsqcup_{J\subseteq[n]}\mathcal B_n(J)\cdot \theta_J$$descends to a spanning set of $\Omega_n/SI^B_n=SR^B_n$.
\end{lemma}
\begin{proof}
    If not, pick a Gale-maximal subset $J\subseteq [n]$ so that $f\cdot \theta_J$ does not belong to the span of $\mathcal B_J$ modulo $SI^B_n$. However, accoding to the previous lemma, $$f\cdot \theta_J\equiv \left(\sum_{m\in\mathcal B_{n}(J)}c_{f,m}\cdot m\cdot\theta_J\right)+\Sigma\pmod{SI^B_n}$$where $\Sigma$ is a linear combination of terms of the form $g\cdot\theta_K$ for $J<_{\text{Gale}}K$. Now the terms $c_{f,m}\cdot m\cdot \theta_J$ lie in the span of $\mathcal B_n$ by definition, and $\Sigma$ lies in the span of $\mathcal B_n$ by the maximality of $J$, which contradicts our assumption.
\end{proof}
Therefore we have, by an application of Lemma \ref{combidentity},
\begin{equation}\label{upperbound}
    \operatorname{Hilb}(SR^B_n;q,z)\le \sum_{J\subseteq [n]}z^{|J|}\sum_{m\in \mathcal B_n(J)}q^{\deg m}=\sum_{k=0}^n [2k]!!_q\cdot \operatorname{Stir}^B_q(n,k)\cdot z^{n-k}.
\end{equation}

\section{A characterization of colon ideals\label{acharacterizationofcolonideals}}
In order to prove that the claimed Hilbert series is also a lower bound for $\operatorname{Hilb}(SR^B_n;q,z)$, we will need to produce sufficiently many linearly independent elements of $SR^B_n$, or equivalently, sufficiently many linearly independent elements of the superharmonic space $SH^B_n$. This will be done in Section \ref{operatortheoremlowerbound}. In preparation for that, we will need to reframe the ideals $(p_{J,1},\ldots, p_{J,n})$ encountered in Section \ref{upperbound} as colon ideals, which we do in this section.

To this end, let us define some more objects. Define a matrix of polynomials 
\begin{equation}
\mathcal H:=\left(h^2_{i-j}(\{i,\ldots, n\})\right)_{\substack{1\le i\le n\\1\le j\le n}}.   
\end{equation}
The notation $h^2_r(S)$ was introduced in Section \ref{upperboundsection}, it refers to the polynomial obtained from the degree $r$ homogeneous symmetric function by setting the variables whose indices do not occur in $S$ to $0$, and replacing the remaining variables by their squares. For the above definition, we extend the definition slightly by setting $h^2_0(S)=1$ and $h^2_r(S)=0$ for any set $S\subseteq[n]$ and any integer $r<0$. As an example, for $n=3$ we have $$\mathcal H=\begin{pmatrix}
    1 & 0 &0\\
    x_2^2+x_3^2& 1 & 0\\
    x_3^4 & x_3^2 & 0
\end{pmatrix}.$$  For $J\subseteq [n]$, define the operators $\mathfrak D_J$ acting on a superspace element $f\in\Omega_n$ as follows: 
\begin{equation}
   \mathfrak D_J(f):=\sum_{|I|=|J|}(-1)^{\sum I}\Delta_{[n]-J,([n]-I)^*}(\mathcal H)\odot d_{2I-1}(f) 
\end{equation}
where $K^*$ denotes reversal: $$K^*:=\{n+1-k:k\in K\}$$ and $2I-1$ denotes $\{2i-1:i\in I\}$. The notation $\Delta_{([n]-J,([n]-I)*}(\mathcal H)$ therefore denotes the matrix minor of $\mathcal H$ with row set $[n]-J$ and column set $([n]-I)*$. 

As an example, let us compute $\mathfrak D_J$ in the case $n=3$ (using $\mathcal H$ seen above) with $J=\{1,3\}$. The possible $I$'s that appear in the sum are $\{1,2\}$, $\{2,3\}$ and $\{1,3\}$. For $I=\{1,2\}$, the row set is $[3]-J=\{2\}$,  the column set is $([3]-I)*=\{3\}^*=\{1\}$, and $2I-1=\{1,3\}$, which leads to the term $$-\Delta_{2,1}(\mathcal H)\cdot d_{13}(f).$$Doing this for each $I$ leads to $$\mathfrak D_{\{1,3\}}(f)=-\Delta_{2,1}(\mathcal H)\cdot d_{13}(f)+\Delta_{2,2}(\mathcal H)\cdot d_{15}(f)-\Delta_{2,3}(\mathcal H)\cdot d_{35}(f).$$Using the appropriate entries from $\mathcal H$ and using $d_i(f)=(x_1^i\odot f)\cdot\theta_1+(x_2^i\odot f)\cdot \theta_2+(x_3^i\odot f)\cdot \theta_3$, the above simplifies to $$\mathfrak D_{\{1,3\}}(f)=(x_1^3x_2^3 + x_1^3x_2x_3^2 - x_1^5x_2-x_1x_2^3x_3^2)\odot f\cdot\theta_1\theta_2-x_1(x_1^2-x_2^2)(x_1^2-x_3^2)x_3\odot f\cdot\theta_1\theta_3.$$

For notational ease, if we let $$\mathfrak F_{J,K}:=\sum_{|I|=|J|=|K|}(-1)^{\sum I}\Delta_{[n]-J,([n]-I)^*}(\mathcal H)\cdot \left|x_k^{2i-1}\right|_{k\in K,i\in I}$$ for subsets $J,K$ of $[n]$, then the above definition can be rewritten into$$\mathfrak D_J(f)=\sum_{|K|=|J|}(\mathfrak F_{J,K}\odot f)\times \theta_K.$$ 

The $\mathfrak D_J$'s defined above may appear complicated, but they satisfy some very convenient properties: they act on $f$ in a Gale-triangular fashion, and the results have simple Gale leading terms. This is going to be useful later for algebraic manipulations, as the triangularity causes many terms to cancel out. We also note that $\mathcal H$ and $\mathfrak D_J$ defined above are in fact natural type B analogs of the corresponding objects defined and used in \cite{rhoades2024hilbert}.

Before we prove these properties of $\mathfrak D_J$, it will be useful to obtain a description of the $\mathfrak F_{J,K}$'s in terms of certain matrix determinants.
\begin{lemma}
    Let $J=\{j_1<\cdots<j_r\}$ and $K=\{k_1<\cdots<k_r\}$ be two subsets of $[n]$ and let $b(J)=(b(J)_1<b(J)_2<\cdots)$ be the entries in $[n]-J$. Define $$A_{J,K}=\begin{pmatrix}
    B_{J,K}\\ C_{J,K}
\end{pmatrix}$$ where $B_{J,K}$ be a $r\times n$ matrix given by $$B_{J,K}=\begin{pmatrix}
    x_{k_1}^{2n-1} & \cdots & x_{k_1}^{2\cdot 1-1}\\
    \vdots & & \vdots\\
    x_{k_r}^{2n-1} &\cdots & x_{k_r}^{2\cdot 1-1}
\end{pmatrix}$$ and $C_{J,K}$ is a $(n-r)\times n$ matrix given by $$C_{J,K}=(h^2_{b(J)_i-j}(x_{b(J)_i},x_{b(J)_i+1},\ldots, x_n))_{\substack{1\le i\le n-r\\1\le j \le n}}.$$Then we have $\mathfrak F_{J,K}=\pm \det(A_{J,K}).$
\end{lemma}
\begin{proof}
    Similar to Lemma 4.5 in \cite{rhoades2024hilbert}, this follows from the the definition of $\mathfrak F_{J,K}$ and the following expansion of the determinant: $$\det(A_{J,K})=\sum_{\substack{I\subset[n]\\|I|=r}}(-1)^{\sum I-\binom{r+1}2}\cdot \Delta_I(B_{J,K})\cdot\Delta_{[n]-I}(C_{J,K})$$where for any matrix $M$, the notation $\Delta_I(M)$ denotes the maximal minor of $M$ with column set $I$.
\end{proof}
For any subset $J\subseteq [n]$, define $f_J\in \mathbb C[\mathbf x_n]$ as $$f_J=\prod_{j\in J}x_j\prod_{j<i\le n}(x_j^2-x_i^2).$$ 
\begin{lemma}\label{GaletriangularityLemma}
    We have $\mathfrak F_{J,K}=0$ unless $J\ge_{\text{Gale}} K$. Furthermore, $$\mathfrak F_{J,J}=\pm f_J.$$
\end{lemma}
\begin{proof}
    Let $K=\{k_1<\cdots<k_r\}$. Multiplying the $i^{\text{th}}$ row of $A_{J,K}$ with $x_{k_i}$, the given statement reduces to proving $$A'_{J,K}=0$$ unless $J\ge_{\text{Gale}}K$ and $A'_{J,J}=\pm \prod_{j\in J}x_j^2\prod_{j<i\le n}(x_j^2-x_i^2)$, where $$A'_{J,K}=\begin{pmatrix}
        x_{k_1}^{2n} & \cdots &x_{k_1}^2\\
        \vdots & &\vdots\\
        x_{k_r}^{2n}&\cdots&x_{k_1}^2\\
        &C_{J,K}
    \end{pmatrix}.$$However, this can be deduced by replacing $x_i$'s with $x_i^2$ in Lemma 4.8 in \cite{rhoades2024hilbert}.
\end{proof}

We are finally ready to describe the ideals $(p_{J,1},\ldots, p_{J,n})$ as colon ideals $(I^B_n:f_J)$. Recall that $I^B_n$ denotes the ideal generated $\mathfrak B_n$-invariants in $\mathbb C[\mathbf x_n]$ with vanishing constant term. To prove this equality of colon ideals,, we will need the following commutative algebra result:

\begin{lemma}[\cite{abe2020hessenberg}, Lemma 2.4]\label{commalg}
     Suppose $\mathfrak a, \mathfrak a'\subset  \mathbb{C}[\mathbf x_n]$ are homogeneous ideals and $f \in \mathbb{C}[\mathbf x_n]$ is a homogeneous polynomial of degree $k$ with $f \not\in\mathfrak  a$. Suppose
$\mathfrak a' \subseteq (\mathfrak a:f )$. If $\mathbb{C}[\mathbf x_n]/{\mathfrak a'}$ is a Poincar\'e duality algebra of socle degree $r$ and $\mathbb{C}[\mathbf x_n]/{\mathfrak a}$ is a Poincar\'e
duality algebra of socle degree $r + k$, then $\mathfrak a' = (\mathfrak a : f )$.
\end{lemma}

In order to apply the above lemma, it will be necessary to prove certain quotients to be Poincar\'e duality algebras. The following lemma, which follows from, for example, \cite[Thm. 6.5.1]{smith1995polynomial}, will be useful for that.
\begin{lemma}\label{PDA}
    If $\mathfrak a\subseteq \mathbb{C}[\mathbf x_n]$ is a homogeneous ideal, a complete intersection, and is generated by the regular sequence $(f_1,\ldots, f_n)$, then $\mathbb{C}[\mathbf x_n]/\mathfrak a$ is a Poincar\'e duality algebra with socle degree $\sum_{i=1}^n(\deg f_i-1)$.
\end{lemma}
\begin{lemma}\label{colonideallemma}
     For any $J\subseteq [n]$, we have the following equality of ideals in $\mathbb C[\mathbf x_n]$: $$(I^B_n:f_J)=(p_{J,1},\ldots, p_{J,n}).$$
\end{lemma}
\begin{proof}
    We will start by showing that $p_{J,i}\in (I^B_n:f_J)$ for any $J\subseteq[n]$ and any $i\in[n]$. It suffices to prove that $p_{J,i}f_J\in I^B_n$, or equivalently, $(p_{J,i}f_J)\odot \delta^B_n=0$. 
    
    Consider $\mathfrak D_J(\delta^B_n)$. This is a linear combination of $d_{2I-1}$ operators with coefficients in $\partial_1,\ldots, \partial _n$, and so by the inclusion $SH'^B_n\subseteq SH^B_n$, $\mathfrak D_J(\delta^B_n)\in SH^B_n$. Also, $q_{J,i}\in SI^B_n$, so we have $$q_{J,i}\odot \mathfrak D_J\left(\delta^B_n\right)=0.$$
    
    We know that $$q_{J,i}=p_{J,i}\cdot\theta_J+\sum_{J<_{\text{Gale}}L}A_L\cdot \theta_L$$and by the previous lemma, we have $$\mathfrak D_J\left(\delta^B_n\right)=\left(f_J\odot\delta^B_n\right)\cdot\theta_J+\sum_{K<_{\text{Gale}}J}B_K\cdot \theta_K$$where $A_K$ and $B_K$'s all lie in $\mathbb C[\mathbf x_n]$. Multiplying these and using the triangularity, we see that $q_{J,i}\odot\mathfrak D_J\left(\delta^B_n\right)=0$ translates into $$p_{J,i}\odot(f_J\odot \delta^B_n)=0$$ which implies the required equality.

    For the reverse containment, it suffices to verify the conditions of Lemma \ref{commalg} where $f=f_J$, $\mathfrak a=I^B_n$ and $\mathfrak a'=(p_{J,1},\ldots, p_{J,n})$. The containment $\mathfrak a'\subseteq (\mathfrak a:f)$ follows from the previous argument. The quotient $\mathbb C[\mathbf x_n]/I^B_n$ is a Poincar\'e duality algebra by Lemma \ref{PDA} has socle degree $n^2$. The quotient $\mathbb C[\mathbf x_n]/\mathfrak (p_{J,1},\ldots, p_{J,n})$ is also a Poincar\'e duality algebra with socle degree $\deg p_{J,1}+\cdots+\deg p_{J,n}-n$, again by Lemma \ref{PDA}. Further, we note that $f_J|\delta^B_N$ for every $J\subseteq [n]$, so to prove $f_J\not\in I^B_n$ it suffices to prove $\delta^B_n\not\in I^B_n$. But this is obvious, since $$\delta^B_n\odot\delta^B_n>0$$ by the properties of an inner product, so in particular $\delta^B_n\odot\delta^B_n\ne 0$.

    Finally, we need to prove the equality of socle degrees. This is $$\left(\sum_{1\le i\le n} \deg p_{J,i}\right)-n+\deg f_J=n^2$$ which, noting that $\deg p_{J,i}=2r_i-[i\in J]$, reduces to $$\left(\sum_{1\le i\le n} 2n-2|J\cup \{i,\ldots, n\}|+2-[i\in J]\right)-n+|J|+\sum_{j\in J}2(n-j)=n^2.$$Clearly $\sum_{1\le i\le n}[i\in J]=|J|.$ Further, in the sum $\sum_{1\le i\le n}|J\cup\{i,\ldots ,n\}|$, each $i\in [n]$ is counted $i$ times, except for those in $J$, which are counted an additional $n-i$ times. Therefore the above equality reduces to $$\left(\sum_{1\le i\le n} 2n+2\right)-2\sum_{1\le i\le n}i-2\sum_{j\in J}(n-j)-n+\sum_{j\in J}2(n-j)=n^2$$ which is clear.
\end{proof}

\section{The operator theorem and lower bound\label{operatortheoremlowerbound}}
\subsection{Operator theorem}
Our next result is an explicit characterization of the superharmonic space $SH^B_n$. The Differential Operator Conjecture (Conj. 1.9 in \cite{swanson2021harmonic}) is about such a characterization for all pseudo-reflection groups of the form $G(m,1,n)$. In \cite{rhoades2024hilbert}, the authors prove the case $m=1$ of this conjecture. Proceeding along similar lines, the following theorem proves the special case $m=2$, which corresponds to the pseudo-reflection group $G(2,1,n)=\mathfrak B_n$.
\begin{theorem}[Operator theorem for type B]\label{operatortheorem}
     The superharmonic space $SH^B_n$ is generated by the elements $d_{2I-1}(\delta^B_n)$ for $I\subseteq[n]$, as a $\mathbb C[\mathbf x_n]$-module under the $\odot$-action: $$SH^B_n=\sum_{I\subseteq[n]}\mathbb C[\mathbf x_n]\odot d_{2I-1}(\delta^B_n).$$
\end{theorem}
\begin{proof}
    Recalling the terminology introduced in Section \ref{harmonicsection}, this is equivalent to showing $SH'^B_n=SH^B_n$. We follow the proof strategy in \cite[Theorem 5.1]{rhoades2024hilbert}. As mentioned before, \cite{swanson2021harmonic} proves that $SH'^B_n\subseteq SH^B_n$, so it suffices to show that $\dim SH^B_n\le \dim SH'^B_n$. But we have the isomorphism $SH^B\cong SR^B_n$, and \ref{upperbound} already gives us an upper bound on $\dim SR^B_n$, so it suffices to show that the same quantity is a lower bound for $\dim SH'^B_n$.

    To this end, note that the discussion preceding Lemma \ref{triangularitylemma} together with Lemma \ref{colonideallemma} implies that for each $J\subseteq [n]$, one can find a set of homogeneous polynomials $\mathcal B_n(J)$ with degree generating function $\prod_{i=1}^n[\deg p_{J,i}]_q$ that descends to a basis of $(I^B_n:f_J)$. We claim that the set $\{g\odot(f_J\odot \delta^B_n):g\in \mathcal B_n(J)\}$ is linearly independent.

    To see why, suppose we have $$\sum_{g\in \mathcal B_n(J)}c_gg\odot(f_J\odot\delta^B_n)=0$$for some constants $c_g$, not all zero. The above implies $$\left(\sum_{g\in\mathcal B_n(J)}c_gg\cdot f_J\right)\odot \delta^B_n\implies \left(\sum_{g\in \mathcal B_n(J)}c_gg\right)\cdot f_J\in I^B_n\implies \sum_{g\in \mathcal B_n(J)}c_gg\in (I^B_n:f_J)$$which contradicts the fact that the images of the $g$'s in $\mathbb C[\mathbf x_n]/(I^B_n:f_J)$ are linearly independent, proving our claim.

    Next, assume that we have $$\left(\sum_{J\subseteq [n]}\sum_{g_J\in \mathcal B_n(J)}c_{J,g_J}(g_J\cdot\theta_J)\right)\odot SH'^B_n=0$$for some choice of $c_{J,g_J}\in\mathbb C$: we claim this may happen only if all the $c_{J,g_J}$'s are zero. Assume to the contrary. Separating out the homogeneous parts of fermionic degree, we may assume the above sum only consists of $J$ of a fixed size $k$, and one of the scalars $c_{J,g_J}$ is non-zero for some $J$ of size $k$. Further, suppose $J_0$ is the Gale-minimal subset of this size for which $c_{J_0,g_{
    {J_0}}}$  is non-zero for some $g_{J_0}$.

    Consider the superspace element $\mathfrak D_{J_0}(\delta^B_n)$. This is clearly in $SH'^B_n$, so the above implies $$\left(\sum_{\substack{J\subseteq [n]\\ |J|=k}}\sum_{g_J\in \mathcal B_n(J)}c_{J,g_J}(g_J\cdot\theta_J)\right)\odot \mathfrak D_{J_0}(\delta^B_n)=0.$$However, notice that in the sum in brackets above, $\theta_J$ has a non-zero coefficient only if $J\ge_{\text{Gale}}J_0$, and in the expansion of $\mathfrak D_{J_0}(\delta^B_n)$, $\theta_J$ has a non-zero coefficient only if $J\le_{\text{Gale}}J_0$ (Lemma \ref{GaletriangularityLemma}). Consequently, when the sum is expanded out, the only surviving terms are those containing $\theta_{J_0}$. The coefficient of $\theta_{J_0}$ in $\mathfrak D_{J_0}(\delta^B_n)$ is $\pm f_J\odot \delta^B_n$, so the above equality becomes $$0=\left(\sum_{g_{J_0}\in \mathcal B_n(J_0)}c_{J_0,g_{J_0}}\cdot g_{J_0}\cdot\theta_{J_0}\right)\odot \left(\pm f_{J_0}\odot\delta^B_n\cdot\theta_{J_0}\right)=\sum_{g_{J_0}\in \mathcal B_n(J_0)}\pm c_{J_0,g_{J_0}}\cdot g_{J_0}\odot(f_{J_0}\odot\delta^B_n)
    $$
    which contradicts our previously proven result that the set $\{g\odot(f_J\odot \delta^B_n):g\in \mathcal B_n(J)\}$ is linearly independent. Thus all the $c_{J,g_J}$'s must be $0$.

    This implies $\dim SH'^B_n$ is at least $\sum_{J}|\mathcal B_n(J)|$. But we have seen before that $\dim SR^B_n$ is at most $\sum_{J}|\mathcal B_n(J)|$ (Lemma \ref{hilbertspanning}), so by the reasoning in the first paragraph of this proof, we are done.
\end{proof}
\subsection{Hilbert series} We are now ready to prove the conjectured Hilbert series for $SR^B_n$. We start by proving a lemma analogous to Lemma 5.2 in \cite{rhoades2024hilbert}.

\begin{lemma}\label{linearindephilbert}
    Suppose for each $J\subseteq [n]$, we have a set $\mathcal B_n(J)$ of homogeneous polynomials that descend to a basis of $\mathbb C[\mathbf x_n]/(I^B_n:f_J)$. Then the set $$\mathcal B_n=\bigsqcup_{J\subseteq[n]}\mathcal B_n(J)\cdot\theta_J$$descends to a linearly independent subset of $SR^B_n$.
\end{lemma}
\begin{proof}
    We follow the same line of reasoning as in Theorem \ref{operatortheorem}. If the stated result is not true, then we can find constants $c_{J,g_J}\in\mathbb C$, not all zero, so that the linear combination $$\sum_{J\subseteq[n]}\sum_{g_J\in \mathcal B_n(J)}c_{J,g_J}g_J\cdot\theta_J$$is in $SI^B_n$, or equivalently it annihilates $SH^B_n$ via the $\odot$-action. As before, we can restrict our attention to $J$'s of a fixed size $k$ by looking at homogenous parts in fermionic degree, and among these $J$'s, pick the Gale-minimal $J_0$ so that $c_{J_0,g_{J_0}}\ne 0$ for some $g_{J_0}\in\mathcal B_n(J_0)$. The above implies $$\left(\sum_{\substack{J\subseteq [n],|J|=k\\J\ge_{\text{Gale}}J_0}}\sum_{g_J\in \mathcal B_n(J)}c_{J,g_J}g_J\cdot\theta_J\right)\odot \mathfrak D_{J_0}(\delta^B_n)=0$$since $\mathfrak D_{J_0}(\delta^B_n)$ belongs to $SH^B_n$. By the same reasoning as in proof for Theorem \ref{operatortheorem}, this reduces to $$\sum_{g_{J_0}\in \mathcal B_n(J_0)}c_{J_0,g_{J_0}}g_{J_0}\odot(f_{J_0}\odot\delta^B_n)=0$$which leads to a contradiction as before.\end{proof}
\begin{theorem}
    The bigraded Hilbert series for $SR^B_n$ is given by $$\operatorname{Hilb}(SR^B_n;q,z)=\sum_{k=0}^n [2k]!!_q\cdot \operatorname{Stir}^B_q(n,k)\cdot z^{n-k}.$$
\end{theorem}

\begin{proof}
    We have seen before that for each $J\subseteq[n]$, one can pick a set of homogeneous polynomials $\mathcal B_n(J)$ satisfying the conditions of Lemma \ref{triangularitylemma} with degree generating function $\prod_{i=1}^n[\deg p_{J,i}]_q.$ Therefore by the above lemma, and the identity shown in Lemma \ref{combidentity} we see that $$\operatorname{Hilb}(SR^B_n;q,z)\ge \sum_{k=0}^n [2k]!!_q\cdot \operatorname{Stir}^B_q(n,k)\cdot z^{n-k}.$$The reverse inequality was established in \ref{upperbound}, and so the result follows.
\end{proof}
Additionally, this yields a procedure for constructing bases of $SR^B_n$ that will be useful in the next section. Combining the results of Lemma \ref{linearindephilbert} and Lemma \ref{hilbertspanning}, we have the following:
\begin{theorem}\label{recipe}
    Suppose for each $J\subseteq [n]$, we have a set $\mathcal B_n(J)$ of homogeneous polynomials that descend to a basis of $\mathbb C[\mathbf x_n]/(I^B_n:f_J)$. Then the set $$\mathcal B_n=\bigsqcup_{J\subseteq[n]}\mathcal B_n(J)\cdot\theta_J$$descends to a basis of $\mathcal SR^B_n$.
\end{theorem}

\section{Hyperplane Arrangements and an Explicit Basis\label{hyperplanebasissection}}

In \cite{angarone2025superspace}, Angarone et al. use the theory of hyperplane arrangements to derive a monomial basis for the superspace coinvariant ring in type $A$. The conjectured monomial basis of $SR^B_n$ (\cite{sagan2024q}) appears to resist similar attacks. However, we can obtain an explicit basis of $SR^B_n$ consisting of homogeneous elements by using hyperplane arrangement. These basis elements have the added property that their bosonic parts completely factor into roots of the type B Weyl group.

For $J\subseteq[n]$ and $i\in [n]$, define the set of polynomials $s_{J,i}$ as follows. Suppose $$\{1,\cdots, i-1\}\setminus J=\{j_1<\cdots <j_r\}.$$ Then we set
\begin{align*}
    s_{J,i}:=\{&1,\\
    &x_i,\\
    &x_i(x_{j_1}+x_i),\;x_i(x_{j_1}^2-x_i^2),\\
    &x_i(x_{j_1}^2-x_i^2)(x_{j_2}+x_i),\;x_i(x_{j_1}^2-x_i^2)(x_{j_2}^2-x_i^2),\\
    &\vdots\\
    &x_i(x_{j_1}^2-x_i^2)\cdots (x_{j_{r-1}}^2-x_i^2)(x_{j_r}+x_i),\;x_i(x_{j_1}^2-x_i^2)\cdots (x_{j_r}^2-x_i^2)\}.
\end{align*}
In other words, we start with $1$, and produce the subsequent elements by multiplying the previous element by the next term in the sequence $$x_i,x_{j_1}+x_i,x_{j_1}-x_i,x_{j_2}+x_i,x_{j_2}-x_i,\ldots, x_{j_r}+x_i,x_{j_r}-x_i.$$

Similarly, define $t_{J,i}$ as 
\begin{align*}
    t_{J,i}:=\{&1,\\
    &(x_{j_1}+x_i),\;(x_{j_1}^2-x_i^2),\\
    &(x_{j_1}^2-x_i^2)(x_{j_2}+x_i),\;(x_{j_1}^2-x_i^2)(x_{j_2}^2-x_i^2),\\
    &\vdots\\
    &(x_{j_1}^2-x_i^2)\cdots (x_{j_{r-1}}^2-x_i^2)(x_{j_r}+x_i),\;(x_{j_1}^2-x_i^2)\cdots (x_{j_r}^2-x_i^2)\}.
\end{align*}which is obtained by performing the previous procedure with the sequence $$x_{j_1}+x_i,x_{j_1}-x_i,x_{j_2}+x_i,x_{j_2}-x_i,\ldots, x_{j_r}+x_i,x_{j_r}-x_i.$$

We aim to show that the set 
\begin{equation}\label{basis}
    \mathcal M=\{p_1p_2\cdots p_n\cdot\theta_J:p_{i}\in s_{j,i}\text{ if }i\not\in J,p_{i}\in t_{J,i}\text{ if }i\in J\text{ where }J\subseteq [n]\}
\end{equation}descends to a basis of $SR^B_n$. Our strategy for the proof is as follows. In Theorem \ref{recipe} we have reduced our problem to finding bases for the quotients of $\mathbb{C}[\mathbf x_n]$ cut out by the colon ideals $I^B_n:f_J$. In this section, we will use Lemma \ref{arrcolonideal} to express these colon ideals as Solomon-Terao ideals of certain free arrangements. Next, we will use an exact sequence closely related to the Addition-Deletion theorem to derive a basis for a family of near-identical hyperplane arrangements. Finally, we will combine these results to obtain the desired basis.

\subsection{Colon ideals as Solomon-Terao ideals}
For ease of notation, set $\alpha_i=H_{x_i}$, $\alpha_{ij}=H_{x_i-x_j}$, and $\overline{\alpha}_{ij}=H_{x_i+x_j}$ for $1\le i<j\le n$. Let $\mathcal B_{\Phi^+}$ denote the set of all hyperplanes corresponding to the positive roots in $\mathfrak B_n$: $$\mathcal B_{\Phi^+}=\{\alpha_{ij},\overline{\alpha}_{ij}:1\le i<j\le n\}\cup\{\alpha_i:1\le i\le n\}.$$ For any $J\subseteq[n]$, define $\mathcal B_J$ to be the following subarrangement of $\mathcal B_{\Phi^+}$:
$$\mathcal B_J=\{\alpha_{ij},\overline{\alpha}_{ij}:j\not\in J,i>j\}\cup\{\alpha_j:j\not\in J\}.$$ 

\begin{lemma}\label{BJfree}
    The arrangement $\mathcal B_J$ is free.
\end{lemma}
\begin{proof}
    We will demonstrate an explicit $\mathbb{C}[\mathbf x_n]$-basis, and then use Saito's criterion (Theorem \ref{saito}) to prove this indeed freely generates $\der(\mathcal B_J)$.

    We define the following elements that will form an $\mathbb{C}[\mathbf x_n]$-basis:

    $$\rho^J_i=\begin{cases}
        \sum_{k=i}^n \left(\prod_{\substack{j\not\in J\\j<i}}(x_j^2-x_k^2)\right)x_k\cdot \partial_k&i\not\in J\\
        \prod_{\substack{j\not\in J\\j<i}}(x_j^2-x_i^2)\cdot\partial_i& i\in J.
    \end{cases}$$

    Now for each $i$, $\rho^J_i$ is homogeonous of degree equal to $$2(\{1,\ldots, i-1\}\setminus J)+[i\not\in J]=\stB_i(J)$$where $[\cdot]$ refers to the Iverson bracket as explained earlier. Note that this is equal to the number of hyperplanes in $\mathcal B_J$ of the form $H_{x_j-x_i},H_{x_j+x_i}$ and $H_{x_i}$. This shows that $\deg \rho^J_1+\cdots+\deg \rho^J_n=|\mathcal B_J|$. Since each $\rho^J_i$ is an $S$-linear combination of $\{\partial_i,\ldots, \partial_n\}$ with a non-zero coefficient for $\partial_i$, which means they are linearly independent over $S$. In order to apply Saito's criterion, it only remains to prove they all belong in $\der(\mathcal B_J)$: \textit{i.e.}, for any hyperplane $H_\alpha$ in $\mathcal B_J$, $\rho^J_i(\alpha)$ is divisible by $\alpha$.\\

    \paragraph{Case 1} Suppose $i\not\in J$. Our first subcase is when $\alpha$ is of the form $x_p\pm x_q$ for $p\not\in J$ and $p<q$.

    \begin{itemize}
        \item If $i\le p<q$, then$$\rho^J_i(x_p\pm x_q)=\prod_{\substack{j\not\in J\\j<i}}(x_j^2-x_p^2)\cdot x_p \pm\prod_{\substack{j\not\in J\\j<i}}(x_j^2-x_q^2)\cdot x_q.$$ Plugging in $x_p=\mp x_q$ makes this vanish, which proves the divisibility.
        \item If $p<i\le q$, then $$\rho^J_i(x_p\pm x_q)=\pm\prod_{\substack{j\not\in J\\j<i}}(x_j^2-x_q^2)\cdot x_q$$which contains $x_p^2-x_q^2$ as a factor since $p\not\in J$ and $p<q$.
        \item Finally, if $p<q<i$, $\rho^J_i(x_p\pm x_q)=0$ which again satisfies the divisibility.
    \end{itemize}

    The second sub-case is when $\alpha$ is of the form $x_p$ for some $p\not\in J$. If $p<i$, $\rho^J_i(x_p)=0$, and if $p\ge i$, $$\rho^J_i(x_p)=\prod_{\substack{j\not\in J\\j<i}}(x_j^2-x_p^2)\cdot x_p.$$ In either case, $\rho^J_i(x_p)$ is divisible by $x_p$.\\

    \paragraph{Case 2} Suppose $i\in J$.The only way $\rho^J_i(\alpha)$ could be non-zero is if $\alpha=x_p-x_q$ for $p\not\in J$ and $p<q$, with $i=q$, in which case $$\rho^J_i(x_p-x_q)=\prod_{\substack{j\not\in J\\j<i}}(x_j^2-x_q^2)\cdot x_q$$which contains $x_p^2-x_q^2$ as a factor, proving the required divisibility.
\end{proof}

We need one last ingredient in order to relate $\mathcal B_J$ to $I^B_n:f_J$.
\begin{lemma}\label{arrcolonideal2}
    Let $\mathcal B\subseteq \mathcal B_{\Phi^+}$ be an arbitrary subarrangement, and define $$\beta_{\mathcal B}=\prod_{H_\alpha\in \mathcal B_{\Phi^+}\setminus\mathcal B}\alpha.$$ Then $\mathcal{ST}(\mathcal B,\mathfrak a)$ is a quotient of $S/(I^B_n)$. Further, if $\mathcal B$ is free, we have $\mathfrak a_{\mathcal B}=(I^B_n:\beta_{\mathcal B})$.
\end{lemma}

This is true for free subarrangements of  any free arrangement, but for our purposes we will only need the special case of subarrangements of $\mathcal B_{\Phi^+}$.
\begin{proof}
    Per Theorem 3.9 in \cite{abe2020hessenberg}, $\mathfrak a (\der(\mathcal B_{\Phi^+}))=I^B_n$. Since $\mathcal B\subseteq \mathcal B_{\Phi^+}$, we have $\der(\mathcal B_{\Phi^+})\subseteq\der(\mathcal B)$ and so $I^B_n\subseteq \mathfrak a(\der(\mathcal B))$, which implies the first statement.

    To prove the second, we will apply Lemma \ref{arrcolonideal}. Applying Lemma \ref{BJfree} to $J=\emptyset$, we see that $\mathcal B_{\Phi^+}$ is free. Also, $\beta_\mathcal{B}|\delta^B_n$, and we have $\delta^B_n\not\in I^B_n$ as seen in the proof of Lemma \ref{colonideallemma}, so $\beta_\mathcal{B}\not\in I^B_n$. Thus the hypotheses of the lemma are satisfied and the result follows.
\end{proof}

Using the above for the free arrangement $\mathcal B_J$, we have:
\begin{corollary}
    For any $J\subseteq [n]$, $\mathcal{ST}(\mathcal B_J,\mathfrak a)=\mathbb{C}[\mathbf x_n]/(I^B_n:f_J).$
\end{corollary}
\subsection{An Inductive Basis}

For ease of induction, we will expand our arrangement $\mathcal B_J$ slightly to $\widetilde{\mathcal B_J}$ defined as
$$\widetilde{\mathcal B_J}=\{\alpha_{ij},\overline{\alpha}_{ij}:j\not\in J,i>j\}\cup\{\alpha_j:1\le j\le n\}.$$
The main result used in this induction is the exact sequence relating the Solomon-Terao algebras of triples of arrangements that satisfy the conditions in Theorem \ref{additiondeletion}.

Consider a triple of arrangements $(\mathcal A,\mathcal A\setminus H,\mathcal A^H)$ where $H\in\mathcal A$ and $\mathcal A$ is an arrangement in $\mathbb C^n$. If $H=H_{\alpha}$ is the zero set of the linear form $\alpha\in \mathbb{C}[\mathbf x_n]$, then we can look at the canonical projection $\pi:\mathbb{C}[\mathbf x_n]\to \overline{\mathbb{C}[\mathbf x_n]}=\mathbb{C}[\mathbf x_n]/(\alpha)$. The derivation module $\der(\mathcal A^H)$ can be thought of as a submodule of $\overline{S}\otimes_{\mathbb C}H_\alpha$.

If $\theta\in\der(\mathcal A)$, then $\theta(\alpha \mathbb{C}[\mathbf x_n])\subseteq \alpha \mathbb{C}[\mathbf x_n]$, which means the derivation $\overline\theta:\overline {\mathbb{C}[\mathbf x_n]}\to\overline{\mathbb{C}[\mathbf x_n]}$ given by $$\overline{\theta}(f+\alpha \mathbb{C}[\mathbf x_n])=\theta(f)+\alpha \mathbb{C}[\mathbf x_n]$$is well-defined. One can show that $\overline{\theta}$ is a derivation of $\mathcal A^H$ \cite[Prop. 4.44]{orlik2013arrangements}.

We have the following exact sequence \[\begin{tikzcd}
0 \arrow[r] & \der(\mathcal A\setminus H) \arrow[r, "\times\alpha"] & \der(\mathcal A) \arrow[r, "\theta\mapsto\overline\theta"] & \der(\mathcal A^H).
\end{tikzcd}\]The rightmost map is not surjective in general, but under specific conditions, we get a surjection letting us extend the sequence as follows:

The map $\theta\mapsto \overline{\theta}$ is surjective under certain conditions, and this gives rise to an exact sequence of derivation modules as follows.
\begin{lemma}[\cite{orlik2013arrangements}]
    Let $(\mathcal A,\mathcal A\setminus H,\mathcal A^H)$ be a triple of arrangements satisfying the conditions of Theorem \ref{additiondeletion}, and suppose $H=H_{\alpha}$.  Then we have an exact sequence of derivation modules 
    \[\begin{tikzcd}
0 \arrow[r] & \der(\mathcal A\setminus H) \arrow[r, "\times\alpha"] & \der(\mathcal A) \arrow[r, "\theta\mapsto\overline\theta"] & \der(\mathcal A^H) \arrow[r] & 0.
\end{tikzcd}\]
\end{lemma}
Finally, we derive an exact sequence of Solomon-Terao algebras for certain subarrangements of $\mathcal B_{\Phi^+}$ that will be crucial in inductively determining $\mathcal{ST}(\widetilde{\mathcal B_J},\mathfrak a)$. This is the type B analog of \cite[Lemma 7.1]{angarone2025superspace}.

\begin{lemma}
    Suppose $\mathcal A$ is a subarrangement of $\mathcal B_{\Phi^+}$ containing the hyperplane $H=H_{\alpha}$, and that the triple $(\mathcal A,\mathcal A\setminus H,\mathcal A^H)$ satisfying the conditions of Theorem \ref{additiondeletion}. Then we have an exact sequence
    \[\begin{tikzcd}
0 \arrow[r] & {\mathcal{ST}(\mathcal A\setminus H,\mathfrak a)} \arrow[r, "\times\alpha"] & {\mathcal{ST}(\mathcal A,\mathfrak a)} \arrow[r] & {\mathcal{ST}(\mathcal A^H,\mathfrak a)} \arrow[r] & 0
\end{tikzcd}\]where the second map is the canonical projections.
\end{lemma}
\begin{proof}
    We use an argument similar to that in \cite{angarone2025superspace}. We have the following exact sequence from the definition of colon ideals:\[\begin{tikzcd}
0 \arrow[r] & \mathbb{C}[\mathbf x_n]/(\mathfrak a_{\mathcal A}:(\alpha)) \arrow[r, "\times\alpha"] & \mathbb{C}[\mathbf x_n]/\mathfrak a_{\mathcal A} \arrow[r] & \mathbb{C}[\mathbf x_n]/(\mathfrak a_{\mathcal A}+(\alpha)) \arrow[r] & 0.\end{tikzcd}\]We note that $\mathbb{C}[\mathbf x_n]/\mathfrak a_{\mathcal A}=\mathcal{ST}(\mathcal A,\mathfrak a)$ by the definition of Solomon-Terao algebras, and $\mathbb{C}[\mathbf x_n]/(\mathfrak a_{\mathcal A}:(\alpha))=\mathcal{ST}(\mathcal A\setminus H)$ by Lemma \ref{arrcolonideal}. To establish the claimed exact sequence, it suffices to show that $\mathbb{C}[\mathbf x_n]/(\mathfrak a_{\mathcal A}+(\alpha))=\mathcal{ST}(\mathcal A^H,\mathfrak a)$.

To see why this holds, we note that the following diagram commutes
\[
\begin{tikzcd}
\der(\mathcal A) \arrow[d, "\mathfrak a"'] \arrow[r, "\theta\mapsto\overline\theta"] & \der(\mathcal A^H) \arrow[d, "\mathfrak a"] \\
\mathbb{C}[\mathbf x_n] \arrow[r, "f\mapsto f+(\alpha)"]                                                 & \mathbb{C}[\mathbf x_n]/(\alpha)                               
\end{tikzcd}
\]where the horizontal maps are both surjections. Therefore $\mathfrak a(\der(\mathcal A^H))$ can be calculated by starting at $\der(\mathcal A)$ in the above diagram and going down, then right, which yields $$\mathfrak a(\der(\mathcal A^H))=\mathfrak a_{\mathcal A}+(\alpha)\subseteq \mathbb{C}[\mathbf x_n]/(\alpha)$$and therefore $$\mathcal{ST}(\mathcal A^H,\mathfrak a)=\mathbb{C}[\mathbf x_n]/(\mathfrak a(\der(\mathcal A^H))=(\mathbb{C}[\mathbf x_n]/(\alpha))/(\mathfrak a_{\mathcal A}+(\alpha))=\mathbb{C}[\mathbf x_n]/(\mathfrak a_{\mathcal A}+(\alpha)),$$which finishes the argument as explained above.
\end{proof}

Since we want to apply the above results to $\widetilde{\mathcal B_J}$, let us also establish the freeness of this arrangement. We will use the same strategy as \ref{BJfree}: using explicit generators of $\der(\widetilde{\mathcal B_J})$ to apply Saito's criterion \ref{saito}.

\begin{lemma}
    The arrangement $\widetilde{\mathcal B_J}$ is free with exponents $(b_{J,1},\cdots, b_{J,n})$ where $$b_{J,i}=2\left|\{1,\ldots, i-1\}\setminus J\right|+1.$$
\end{lemma}

\begin{proof}
    As in the proof of Lemma \ref{BJfree}, we will use Saito's criterion by demonstrating an explicit set of generators of $\der(\widetilde{\mathcal B_J})$.

    Let $$\mu^J_i=
        \sum_{k=i}^n \left(\prod_{\substack{j\not\in J\\j<i}}(x_j^2-x_k^2)\right)x_k\cdot \partial_k
    $$

    Clearly each $\mu^J_i$ is homogeneous of degree $2|\{1,\ldots, i-1\}\setminus J|+1$ which equals the number of hyperplanes in $\widetilde{\mathcal B_J}$ of the form $H_{x_j-x_i}$, $H_{x_j+x_i}$ or $H_{x_i}$, which implies $\deg \mu^J_1+\cdots+\mu^J_n$ equals to the total number of hyperplanes, $|\widetilde{\mathcal B_J}|$. Their $\mathbb{C}[\mathbf x_n]$-linear independence follows by triangularity: each $\mu^J_i$ is an $\mathbb{C}[\mathbf x_n]$-linear combination of $\partial_i,\ldots,\partial_n$ with non-zero coefficients.

    To prove they all belong to $\der(\widetilde{\mathcal B_J})$, we show that $\alpha|\mu^J_i(\alpha)$ for any hyperplane $H_\alpha$ in $\widetilde{\mathcal B_J}$. 

    \begin{itemize}
        \item Suppose $\alpha=x_p$ for some $p$. If $p<i$, $\mu^J_i(x_p)=0$, and if $p\ge i$, $$\mu^J_i(x_p)=\prod_{\substack{j\not\in J\\j<i}}(x_j^2-x_p^2)\cdot x_p,$$ and in either case, $x_p|\mu^J_i(x_p)$.

        \item Suppose $\alpha=x_p\pm x_q$: we must have $p\not \in J$ and $p<q$.  If $i\le p<q$, then$$\mu^J_i(x_p\pm x_q)=\prod_{\substack{j\not\in J\\j<i}}(x_j^2-x_p^2)\cdot x_p \pm\prod_{\substack{j\not\in J\\j<i}}(x_j^2-x_q^2)\cdot x_q.$$ This vanishes for $x_p=\mp x_q$ which proves the divisibility. If instead $p<i\le q$, then $$\rho^J_i(x_p\pm x_q)=\pm\prod_{\substack{j\not\in J\\j<i}}(x_j^2-x_q^2)\cdot x_q$$which is evidently divisible by $x_p^2-x_q^2$. Finally, if $p<q<i$, then $\mu^J_i(x_p\pm x_q)=0$, and the conclusion follows.
    \end{itemize}
\end{proof}

We are now in a position to obtain a basis of $\mathcal{ST}(\widetilde{\mathcal B_J},\mathfrak a)$ by using the exact sequence deduced above, which subsequently yields a basis for the Solomon-Terao algebra associated with the closely related arrangement $\mathcal B_J$.
\begin{lemma}
    For any $J\subseteq [n]$, the set $$\widetilde{\mathcal M}(J)=\{p_1p_2\cdots p_n: p_i\in s_{J,i}\}$$descends to a basis of $\mathcal{ST}(\widetilde{\mathcal B_J},\mathfrak a)$.
\end{lemma}
\begin{proof}
    We will use induction on $n$. Suppose $[n-1]\setminus J=\{j_1<\cdots<j_r\}$. We let $\widetilde{\mathcal B_{J'}}$ denote the arrangement associated to $J'=J\setminus\{\alpha_{in},\overline\alpha_{in},\alpha_n:1\le i< n\}$ considered as an arrangement in $\mathbb C^{n-1}$.

    Let us also define the sequence of arrangements $$\mathcal B_0,\mathcal B_1,\mathcal B_2,\ldots, \mathcal B_{2r+1}$$where $\mathcal B_i$ is formed from $\widetilde{\mathcal B_J}$ by deleting the first $i$ hyperplanes in the sequence $$\alpha_n,\overline\alpha_{j1n},\alpha_{j1,n},\overline\alpha_{j2,n},\alpha_{j2,n},\ldots,\overline\alpha_{jr,n},\alpha_{jr,n}.$$ Note that $\mathcal B_0=\widetilde{\mathcal B_J}$.
    we claim that $\mathcal B_i$ is free with exponents $(b_{J,1},\ldots, b_{J,n-1},b_{J,n}-i)$. This is clear for $i=0$. Assuming this is true for some $i$, consider the triple $(\mathcal B_i,\mathcal B_{i+1},\widetilde{\mathcal B_{J'}})$. Suppose $\mathcal B_{i+1}$ is obtained from $\mathcal B_i$ by deleting the hyperplane $\alpha$. One can check that $(\mathcal B_{i+1})^{\alpha}=\widetilde{\mathcal B_{J'}}$. Further, $\widetilde{\mathcal B_{J'}}$ is free with exponents $(b_{J,1},\cdots, b_{J,n-1})$, so this triple satisfies the hypotheses of Theorem \ref{additiondeletion} and thus $\mathcal B_{i+1}$ must be free with exponents $(b_{J,1},\ldots, b_{J,n}-(i+1))$. This gives rise to the short exact sequence $$0\longrightarrow\mathcal{ST}(\mathcal B_{i+1},\mathfrak a)\stackrel{\times f}{\longrightarrow}\mathcal{ST}(\mathcal B_{i},\mathfrak a)\longrightarrow \mathcal{ST}
    \left(\widetilde{\mathcal B_{J'}},\mathfrak a\right)\longrightarrow 0$$where $\alpha=H_f$.

    By induction hypothesis, $\mathcal{ST}(\widetilde{\mathcal B_{J'}},\mathfrak a)$ has a basis consisting of images of polynomials of the form $p_1\cdots p_{n-1}$ where $p_i\in s_{J,i}$. Further, note that the same polynomials form a basis for $\mathcal {ST}(\mathcal B_{2r+1},\mathfrak a)$: indeed, $\mathcal B_{2r+1}$ has the same hyperplanes as $\widetilde{\mathcal B_{J'}}$, which means a basis $\der(\mathcal B_{2r+1})$ is obtained by taking a basis of $\der(\widetilde{\mathcal B_{J'}})$ and adding $\partial_n$. Therefore 
    \begin{multline*}
        \mathcal{ST}(\mathcal B_{2r+1},\mathfrak a)=\mathbb C[x_1,\ldots, x_n]/(\mathfrak a(\der(\widetilde{\mathcal B_{J'}}))+(x_n))\\
        \cong \mathbb C[x_1,\ldots, x_{n-1}]/(\mathfrak a(\der(\widetilde{\mathcal B_{J'}}))=\mathcal{ST}(\widetilde{\mathcal B_{J'}},\mathfrak a).
    \end{multline*}

    Now the above exact sequence for $i=2r$ implies 
    polynomials of the form $p_1\cdots p_{n-1}q$  where $p_i\in s_{J,i}$ and $q\in \{1,x_{j_r}-x_i\}$ form a basis of $\mathcal{ST}(\mathcal B_{2r},\mathfrak a)$. Using this and the above exact sequence for $i=2r-1$, we see that polynomials of the form $p_1\cdots, p_{n-1}q$ where $p_i\in s_{J,i}$ and $q\in \{1,x_{j_r}+x_i,x_{j_r}^2-x_{i}^2\}$ yield a basis for $\mathcal{ST}(\mathcal B_{2r-1},\mathfrak a)$. Iterating this process, we eventually arrive at the basis $\{p_1\cdots p_nq:p_i\in s_{J,i},q\in s_{J,n}\}$ for $\mathcal{ST}(\mathcal B_0,\mathfrak a)=\mathcal{ST}(\widetilde{\mathcal B_{J}},\mathfrak a)$, which proves our claim.
\end{proof}
\begin{theorem}
    The set of polynomials $$\mathcal M(J)=\{p_1p_2\cdots p_n:p_{j,i}\in s_{j,i}\text{ if }i\not\in J,p_{J,i}\in t_{J,i}\text{ if }i\in J\}$$descend to a basis for $\mathbb{C}[\mathbf x_n]/(I^B_n:f_J)$.
\end{theorem}
\begin{proof}
    The claimed basis for $\mathcal{ST}(\mathcal B_J,\mathfrak a)=\mathbb{C}[\mathbf x_n]/(I^B_n:f_J)$ agrees with the expected degree generating function $\prod_{i=1}^n[\deg p_{J,i}]_q$, so it suffices to show this set is linearly independent. To this end, note that $$\mathcal{ST}(\widetilde{\mathcal B_J},\mathfrak a)$$is equal to $\mathbb{C}[\mathbf x_n]/(I^B_n:\widetilde{f_J})$ where $$\widetilde{f_J}=\prod_{j<i\le n}(x_j^2-x_i^2).$$This follows from the fact that $\widetilde{\mathcal B_J}$ is free and Lemma \ref{arrcolonideal}. Therefore we have an injection $$\mathcal{ST}(\mathcal B_J,\mathfrak a)=\mathbb{C}[\mathbf x_n]/(I^B_n:f_J)\xrightarrow{\times\prod_{j\in J}x_j}\mathbb{C}[\mathbf x_n](I^B_n:\widetilde{f_J})=\mathcal{ST}(\widetilde{\mathcal B_J},\mathfrak a).$$This injection maps $\mathcal M(J)$ to a subset of $\widetilde{\mathcal M}(J)$, and therefore $\mathcal M(J)$ is linearly independent as claimed.
\end{proof}
Now Theorem \ref{recipe} may be combined with the above result to yield the following:
\begin{corollary}\label{hyperplanebasiscor}
    The set $\mathcal M\subseteq \Omega_n$ defined in \eqref{basis} descends to a basis of $SR^B_n$.
\end{corollary}
\section{Conclusion\label{conclusion}}
In this paper, we derived the Hilbert series for $SR^B_n$ and constructed an explicit basis from images of easily factorable superspace polynomials. There are several open problems that would help refine these results.

In \cite{sagan2024q} the authors conjectured that the set of monomials $\{x_1^{a_1}\cdots x_n^{a_n}\cdot\theta_J:a_i\le \stB_i(J),\;J\subseteq[n]\}$ in $\Omega_n$ descend to a basis of $SR^B_n$. This is still unproven; the techniques used to prove the type A analogue in \cite{angarone2025superspace} do not appear to easily extend to yield monomial basis elements for the type B case. The leading monomials of the explicit basis (in lexicographical order with $x_1<x_2<\cdots<x_n$) are exactly the elements in the conjectured monomial basis; however, there does not seem to be an obvious way to derive one from the other.

Questions on the representation-theoretic structure of $SR^B_n$ as a $\mathfrak B_n$-module remain largely unanswered. The analogous situation for type A is closely connected to a celebrated conjecture made by the Fields Institute Combinatorics Group \cite{zabrocki2019module} -- as explained in \cite{rhoades2022set}, the aforementioned conjecture implies the $\mathfrak S_n$-module isomorphism 
\begin{equation}
    SR_n\cong \text{sign}\otimes \mathbb C[\mathcal{OP}_n]
\end{equation}
where $\text{sign}$ is the one-dimensional sign representation, and $\mathcal{OP}_n$ is the family of ordered set partitions of $[n]$ equipped with a natural $\mathfrak S_n$-action. 

In private communication, Rhoades conjectured that a similar statement holds for the type B case: we have an isomorphism of $\mathfrak B_n$-modules \begin{equation}\label{isoB}
    SR^B_n\cong \text{sign}^B\otimes \mathbb C\left[\mathcal{OP}_n^B\right]
\end{equation} where $\text{sign}^B$ is the sign representation of $\mathfrak B_n$ and $\mathcal{OP}^B_n$ denotes the family of \textit{signed} ordered set partitions (see \cite{sagan2024q}) for a definition) of $\{-n,-n+1,\ldots, n-1,n\}$. 

The results in this paper confirm that the spaces on both sides of \ref{isoB} above have identical Hilbert series. As such, it would suffice to construct injective (or surjective) $\mathfrak B_n$-module homomorphisms in order to establish the isomorphism in \ref{isoB}.
\section{Acknowledgements}
The author is grateful to Brendon Rhoades for many helpful conversations and suggestions.
\bibliographystyle{abbrv} 
\bibliography{refs} 
\end{document}